\documentclass[a4paper,reqno]{amsart}

\usepackage{amsmath,amssymb,amsthm}
\usepackage{mathrsfs}
\usepackage{indentfirst,color}
\usepackage[colorlinks,citecolor=red,linkcolor=blue,urlcolor=cyan]{hyperref}
\usepackage{graphicx}
\usepackage{tikz,tikz-cd}
\usepackage{bm}

\theoremstyle{plain}%
\newtheorem{theorem}{Theorem}[section]%
\newtheorem{conjecture}[theorem]{Conjecture}%
\newtheorem{proposition}[theorem]{Proposition}%
\newtheorem{lemma}[theorem]{Lemma}%
\newtheorem{corollary}[theorem]{Corollary}%
\newtheorem{example}[theorem]{Example}%

\theoremstyle{definition}%
\newtheorem{definition}[theorem]{Definition}%
\theoremstyle{remark}%
\newtheorem{remark}[theorem]{Remark}%

\newcommand{\CC}{{\mathbb{C}}}%
\newcommand{\NN}{{\mathbb{N}}}%
\newcommand{\RR}{{\mathbb{R}}}%

\makeatletter
\newsavebox{\@brx}
\newcommand{\llangle}[1][]{\savebox{\@brx}{\(\m@th{#1\langle}\)}%
    \mathopen{\copy\@brx\kern-0.5\wd\@brx\usebox{\@brx}}}
\newcommand{\rrangle}[1][]{\savebox{\@brx}{\(\m@th{#1\rangle}\)}%
    \mathclose{\copy\@brx\kern-0.5\wd\@brx\usebox{\@brx}}}
\makeatother

\newcommand{\Id}{{\operatorname{Id}}}%
\newcommand{\Vol}{{\operatorname{Vol}}}%
\newcommand{\supp}{\operatorname{supp}}%
\newcommand{\image}{\textup{i}}%
\newcommand{\Bl}{\operatorname{Bl}}
\newcommand{\End}{\operatorname{End}}
\newcommand{\Hom}{\operatorname{Hom}}

\newcommand{\rk}{\operatorname{rank}}
\newcommand{\codim}{\operatorname{codim}}
\newcommand{\tr}{\operatorname{tr}}

\newcommand{\BC}{\mathrm{BC}}
\newcommand{\HN}{\mathrm{HN}}
\newcommand{\VII}{\mathrm{VII}}
\newcommand{\Pic}{\mathrm{Pic}}

\newcommand{\lin}{\mathrm{lin}}
\newcommand{\diag}{\mathrm{diag}}
\newcommand{\FS}{\textup{FS}}%
\renewcommand{\bar}[1]{\overline{#1}}%
\newcommand{\ddbar}{\image\partial\bar{\partial}}%
\renewcommand{\leq}{\leqslant}%
\renewcommand{\geq}{\geqslant}%

\numberwithin{equation}{section}

\begin{document}

\title[Uniform RC-positivity of tangent bundles]{Uniform RC-positivity of tangent bundles}

\author[Chenghao Qing]{Chenghao Qing}
\address{Chenghao Qing: Yau Mathematical Sciences Center, Tsinghua University, Beijing 100084, China.}
\email{qingchenghao@amss.ac.cn}

\author[Hongzhao Sun]{Hongzhao Sun}
\address{Hongzhao Sun: School of Mathematical Sciences, Peking University, Beijing, 100871, China.}
\email{sunhongzhao@amss.ac.cn}

\author[Xiangyu Zhou]{Xiangyu Zhou}
\address{Xiangyu Zhou: Institute of Mathematics, State Key Laboratory of Mathematical Sciences, AMSS, Chinese Academy of Sciences, Beijing 100190, China.}
\email{xyzhou@math.ac.cn}

\date{}

\thanks{
The first author was supported by the China Postdoctoral Science Foundation (Grant No. 2025M773087).
The third author was supported by National Key R\&D Program of China (No. 2021YFA1003100) and by the National Natural Science Foundation of China (Grant No. 12288201).}

\begin{abstract}
    In this paper, we prove that every rationally connected projective manifold admits a smooth uniformly RC-positive Hermitian metric on its tangent bundle, answering Yang's question and giving a characterization of rational connectedness by uniform RC-positivity. 
    We find an example to show that RC-positivity alone does not characterize rational connectedness. 
    We also prove that uniform RC-positivity of the tangent bundle is preserved under blow-ups along connected smooth centers on compact complex manifolds.
    In the non-K\"ahler setting, 
    we construct such metrics on all Hopf and Kato surfaces and obtain classification results for compact complex surfaces.
    
\end{abstract}

\keywords{Rational connectedness, RC-positivity, Hermitian metrics.}

\subjclass[2020]{
    32Q10, 
    14M22, 
    53C55, 
    32Q57, 
    32J15  
}

\maketitle

\section{Introduction}

Curvature positivity of the tangent bundle imposes strong restrictions on the geometry of a compact complex manifold.
Rational connectedness is an algebro-geometric property defined by the existence of rational curves.
A projective manifold is called rationally connected if any two points can be connected by some rational curve.
A central question is to identify curvature conditions that characterize this property. 
Yang's notions of RC-positivity and uniform RC-positivity \cite{Yan18,Yan20} provide a natural setting for this question.

\begin{definition}
    Let $(E,h)$ be a Hermitian holomorphic vector bundle over a complex manifold $X$.
    \begin{enumerate}
        \item $(E,h)$ is called uniformly RC-positive at a point $q\in X$ if there exists a nonzero vector $u\in T_qX$ such that for every nonzero vector $v\in E_q$, 
        $$R^h(u,\bar{u},v,\bar{v})>0.$$
        \item $(E,h)$ is called RC-positive at a point $q\in X$ if for every nonzero vector $v\in E_q$, there exists a nonzero vector $u\in T_qX$ such that
        $$R^h(u,\bar{u},v,\bar{v})>0.$$
    \end{enumerate}
    The bundle $(E,h)$ is called (uniformly) RC-positive if it is (uniformly) RC-positive at each point $q\in X$.
\end{definition}

Let $X$ be a compact complex manifold. 
A holomorphic vector bundle $E$ over $X$ is said to be (uniformly) RC-positive if it admits a smooth Hermitian metric $h$ such that $(E,h)$ is (uniformly) RC-positive.

Yang proved in \cite{Yan20} that a compact K\"ahler manifold with uniformly RC-positive tangent bundle is projective and rationally connected.
Since the tangent bundle of a compact K\"ahler manifold with positive holomorphic sectional curvature is uniformly RC-positive (\cite[Theorem 5.1]{Yan20}),
this confirms a conjecture of Yau (\cite[Problem 47]{Yau82}) that a compact K\"ahler manifold with positive holomorphic sectional curvature is a projective and rationally connected manifold.
In the converse direction, Yang asked whether rational connectedness could be characterized by  RC-positivity or uniform RC-positivity of the tangent bundle.

\begin{conjecture}[{\cite[Problem 4.15]{Yan20}}] \label{Conjecture:Yang's conjecture}
    The following statements are equivalent on a projective manifold $X$:
    \begin{enumerate}
        \item $X$ is rationally connected.
        \item $TX$ is RC-positive.
        \item $TX$ is uniformly RC-positive.
    \end{enumerate}
\end{conjecture}

It's known that (3) implies (1) and (2). In the present paper, our first main result is to establish or disprove all the converse parts, solving Conjecture~\ref{Conjecture:Yang's conjecture} completely.

\begin{theorem}[=Theorem~\ref{Main Thm 0'}]\label{Main Thm 0}
	Let $X$ be a compact K\"ahler manifold.
	If $X$ is rationally connected, then $TX$ is uniformly RC-positive.
\end{theorem}

We discover a counterexample to the equivalence between rational connectedness and RC-positivity of the tangent bundle in Section~\ref{Sec:Pre}.
More precisely, we show that the Fermat quartic $K3$ surface has an RC-positive tangent bundle for the metric induced by the Fubini--Study metric,
but it is not rationally connected and its tangent bundle is not uniformly RC-positive.

\

\textit{Remark}. A related characterization was obtained by Li--Zhang--Zhang (\cite[Theorem 1.7]{LZZ25}). 
They proved that a compact K\"ahler manifold $X$ is projective and rationally connected if and only if its tangent bundle $TX$ is mean curvature positive. 
This means that there exist a Hermitian metric $\omega$ on $X$ and a Hermitian metric $h$ on $TX$ such that
$i\Lambda_\omega\Theta_{TX,h}>0$.
Their theorem confirms a question of Demailly and Yang (\cite[Problem 4.17]{Yan20}).
Note that mean curvature positivity is an intermediate concept between uniform RC-positivity and RC-positivity (see \cite[Theorem 3.6]{Yan18} and \cite[Remark 1.12]{LZZ25}).

\

\textit{Remark}. In complex dimension two, Zhang proved that a compact complex surface is rational if and only if it admits a K\"ahler metric with positive holomorphic sectional curvature \cite{Zha26}.
Since rationality is equivalent to rational connectedness for projective surfaces, this establishes the equivalence between rational connectedness and uniform RC-positivity of the tangent bundle for smooth projective surfaces, which was still open in higher dimension before this paper.

\

Note that the proof of Theorem~\ref{Main Thm 0} in dimension at least three is based on the existence of an embedded very free curve on rationally connected manifolds.
The algebraic property plays an essential role in the analytic construction of the required Hermitian metric.
In dimension two, we give a proof independent of Zhang's theorem by using the blow-up result, Theorem~\ref{Main Thm 1}, below.

Over the years, numerous characterizations of rationally connected manifolds have been established through both algebraic and differential geometric approaches (\cite{Kol96,AK03,CDP15,Yan20,LZZ25}). We collect them together with Theorem~\ref{Main Thm 0} as follows.

\begin{theorem}
    Let $X$ be a smooth connected complex projective variety of dimension $n$, then the following properties are equivalent:
    \begin{enumerate}
        \item $X$ is rationally connected.
        \item There exists a very free rational curve $f:\mathbb P^1\rightarrow X$, that is, $f^*TX$ is ample on $\mathbb{P}^1$.
        \item For every invertible subsheaf $\mathcal{F}\subset\Omega_X^p$, $1\leq p \leq n$, $\mathcal{F}$ is not pseudo-effective.
        \item For every invertible subsheaf $\mathcal{F}\subset\mathcal{O}((T^*X)^{\otimes p})$, $p\geq 1$, $\mathcal{F}$ is not pseudo-effective.
        \item For some (resp. for any) ample line bundle $A$ on $X$, there exists a constant $C_A>0$ such that
        $$H^0(X, (T^*X)^{\otimes m} \otimes A^{\otimes k}) = 0 \quad \textit{for all } m, k\in\mathbb{N}^* \textit{ with } m \geq C_A k.$$
        \item There exist a movable curve $C$ and a constant $\delta>0$, such that for any coherent analytic quotient sheaf $\mathcal{Q}$ of $TX$, $c_1(\mathcal{Q}) \cdot [C] \geq\delta\cdot\rk(\mathcal{Q})$.
        \item The tangent bundle $TX$ is mean curvature positive.
        \item The tangent bundle $TX$ is HN-positive (see Section~\ref{Sec:Pre} for the precise definition).
        \item The line bundle $\mathcal{O}_{\Lambda^p T^*X}(-1)$ is uniformly RC-positive for every $1\leq p\leq n$.
        \item The tangent bundle $TX$ is uniformly RC-positive.
    \end{enumerate}
\end{theorem}

Our second main result concerns blow-ups and holds without a K\"ahler assumption.
Rational connectedness is preserved under blow-ups of projective manifolds.
Then by Theorem~\ref{Main Thm 0}, the uniform RC-positivity is also preserved in that setting.
The following theorem extends this conclusion to compact complex manifolds by a direct curvature construction.
\begin{theorem}[=Theorem~\ref{Main Thm 1'}]\label{Main Thm 1}
    Let $X$ be a compact complex manifold and let $Y\subset X$ be a connected closed complex submanifold. 
    Let $\pi:\widetilde{X}:=\Bl_YX\longrightarrow X$ be the blow-up of $X$ with center $Y$.
    If $TX$ admits a smooth uniformly RC-positive Hermitian metric, then
    $T\widetilde{X}$ is also uniformly RC-positive.
\end{theorem}

Based on this result, we consider the classification of compact complex surfaces by using uniform RC-positivity and the Enriques--Kodaira classification, especially in the non-K\"ahler case.
The K\"ahler case is determined by Theorem~\ref{Main Thm 0}.

\begin{theorem}[{=Theorem~\ref{Main Thm 2'}}] \label{Main Thm 2}
    Let $X$ be a compact connected complex surface.
    \begin{enumerate}
        \item If $TX$ is uniformly RC-positive, then $\kappa(X)=-\infty$, and the minimal model of $X$ is either rational or of class VII.  
        An irrational ruled surface cannot occur.
        \item If $X$ is K\"ahler, then
        $TX$ is uniformly RC-positive if and only if $X$ is rational.
        \item Every Hopf surface has uniformly RC-positive tangent bundle.  
        The same holds after any finite sequence of point blow-ups.
        \item No Inoue-Bombieri surface, and no finite sequence of point blow-ups of one, has uniformly RC-positive tangent bundle.
        \item Among minimal class VII surfaces with $b_2=0$, the uniformly RC-positive ones are exactly the Hopf surfaces.  
    \end{enumerate}
\end{theorem}

We finally consider surfaces of class $\VII_0^+$, namely minimal class $\VII$ surfaces with  positive second Betti number.

\begin{theorem} [=Theorem~\ref{Main Thm 3'}] \label{Main Thm 3}
    Let $X$ be a surface of class $\VII_0^+$. Then
    \begin{enumerate}
        \item $-K_X$ is (uniformly) RC-positive.
        \item If $X$ is a Kato surface, then $TX$ is
        uniformly RC-positive.
        \item If $X$ carries no holomorphic foliation, then $TX$ is mean curvature positive.
    \end{enumerate}
\end{theorem}

The rest of the paper is organized as follows. 
In Section~\ref{Sec:Pre}, we recall the concept of uniform RC-positivity and its geometric properties.
We also compute the positivity of the Fermat quartic $K3$ surface.
In Section~\ref{Sec:Equivalence}, we prove Theorem~\ref{Main Thm 0}, the equivalence between rational connectedness and uniform RC-positivity of tangent bundle, by an embedded very free curve.
In Section~\ref{Sec:Blow-up}, we derive Theorem~\ref{Main Thm 1} by a conformal change.
In Section~\ref{Sec:Classification 1}, we compute the curvature of general Hopf surfaces and prove Theorem~\ref{Main Thm 2}.
In Section~\ref{Sec:Classification 2}, we consider the case of class $\VII_0^+$ surfaces and give some partial results (Theorem~\ref{Main Thm 3}) on the positivity of their anticanonical bundles and tangent bundles.

\section{Preliminaries and a counterexample} \label{Sec:Pre}

\subsection{Curvature of vector bundles and rational connectedness}

Let $(E,h)$ be a holomorphic vector bundle on a complex manifold $X$ equipped with a smooth Hermitian metric $h$ and the corresponding Chern connection $D_E$.
The Chern curvature of $(E,h)$ 
$$i\Theta_{E,h}\in C^\infty(X,\Lambda^{1,1}T^*X\otimes\End(E))$$
defines the curvature tensor
$$R^h\in C^\infty(X,\Lambda^{1,1}T^*X\otimes E^*\otimes\overline{E}^*)$$
to be
$$R^h(u,\bar{v},e,\bar{f})=\langle i\Theta_{E,h}(u,\bar{v})e,\bar{f} \rangle_h$$
for tangent vectors $u,v\in TX$ and vectors $e,f\in E$.

Let $\{z^i\}_{i=1}^n$ be local holomorphic coordinates on $X$ and $\{e_\alpha\}_{\alpha=1}^r$ a local frame of $E$. 
The curvature tensor $R^h$ has components
$$R^h_{i\bar{j}\alpha\bar{\beta}}=
-\frac{\partial^2h_{\alpha\bar{\beta}}}{\partial z^i\partial\bar{z}^j}
+h^{\gamma\bar{\delta}}\frac{\partial h_{\alpha\bar{\delta}}}{\partial z^i}\frac{\partial h_{\gamma\bar{\beta}}}{\partial\bar{z}^j}.$$

For a holomorphic subbundle $S$ of $(E,h)$ and its quotient vector bundle $Q:=E/S$, we consider the exact sequence
$$0\longrightarrow S\longrightarrow E\longrightarrow Q\longrightarrow 0.$$
Denote by $h_S$ and $h_Q$ the induced and quotient metrics on $S$ and $Q$ respectively, and by $D_S$ and $D_Q$ the corresponding Chern connections.
We have the $C^\infty$ isomorphism $E\cong S\oplus Q$.
With respect to this decomposition, $D_E$ can be written as
\begin{equation*}
    D_E=
    \begin{pmatrix}
        D_S & -\beta^* \\
        \beta & D_Q
    \end{pmatrix}
\end{equation*}
where $\beta\in C^\infty(X,\Lambda^{1,0}T^*X\otimes\Hom(S,Q))$ is called the second fundamental form of $S$ in $E$ and where $\beta^*\in C^\infty(X,\Lambda^{0,1}T^*X\otimes\Hom(Q,S))$ is the adjoint of $\beta$ satisfying $\bar\partial_{\Hom(Q,S)} \beta^*=0$
(see \cite[Chapter V, \S 14]{Dem}).

Then $i\Theta_{S,h_S}$ and $i\Theta_{Q,h_Q}$ can be given in terms of $i\Theta_{E,h}$ by the following formulas:
\begin{equation*}
    i\Theta_{S,h_S}=i\Theta_{E,h}|_S+i\beta^*\wedge \beta \quad \text{and} \quad
    i\Theta_{Q,h_Q}=i\Theta_{E,h}|_Q+i\beta\wedge \beta^*,
\end{equation*}
where $i\Theta_{E,h}|_S$ and $i\Theta_{E,h}|_Q$ denote the blocks in the matrix of $i\Theta_{E,h}$ corresponding to $\Hom(S,S)$ and $\Hom(Q,Q)$.
Therefore, for a tangent vector $u\in TX$, a vector $e\in S$, and a vector $f\in Q$, we have 
\begin{equation}\label{Formula:curvature of subbundle}
    \begin{aligned}
        R^h(u,\bar{u},e,\bar{e})-\langle \beta(u)(e),\beta(u)(e)\rangle_{h_Q}&=R^{h_S}(u,\bar{u},e,\bar{e}), \\
        R^h(u,\bar{u},f,\bar{f})+\langle \beta^*(\bar{u})(f),\beta^*(\bar{u})(f)\rangle_{h_S}&=R^{h_Q}(u,\bar{u},f,\bar{f}).
    \end{aligned}
\end{equation}

\

Inspired by the proof of \cite[Theorem 1.6]{Yan20}, we will use conformal changes repeatedly in the proofs of the main results.

\begin{lemma}
    For a smooth function $F$ and Hermitian metric $H=e^{-F}h$, 
    the curvature tensor $R^H$ has the expression
    $$R^H=e^{-F}(R^h+\ddbar F\otimes h).$$
\end{lemma}

We now recall some basic properties of uniformly RC-positive metrics which will be used in this paper.
\begin{proposition}[{\cite[Proposition 2.9]{Yan20}}]\label{Proposition:Constants in RC-positive}
    Let $(E,h)$ be uniformly RC-positive on a neighborhood of a compact set $K\subset X$.
    Then for any Hermitian metric $g$ on $TX$, there exists a constant $C_1=C_1(K,h,g)>0$ such that for any point $q\in K$, there exists a $g$-unit vector $u\in T_qX$ such that
    $$R^h(u,\bar{u},v,\bar{v})\geq C_1|v|^2_h \quad \text{for every } v\in E_q.$$
    More precisely, 
    $$C_1:=\inf_{q\in K}\sup_{u\in T_qX\setminus\{0\}}\inf_{v\in E_q\setminus\{0\}}\frac{R^h(u,\bar{u},v,\bar{v})}{|u|^2_g|v|^2_h}>0.$$
    On the other hand, there exists a constant $0<C_2<+\infty$ such that for any point $q\in K$, every $g$-unit vector $u\in T_qX$ and every $v\in E_q$,
    $$R^h(u,\bar{u},v,\bar{v})\geq-C_2|v|^2_h.$$
\end{proposition}

\begin{proposition}[{\cite[Corollary 2.11, Proposition 2.12]{Yan20}}] \label{Proposition:Basic property for uniformly RC-positive vector bundle}
    Let $(E,h)$ be a uniformly RC-positive vector bundle over a compact complex manifold $X$.
    Then 
    \begin{enumerate}
        \item $(E,h)$ is RC-positive;
        \item $E^{\otimes m}$, $S^kE$ and $\Lambda^pE$ are all uniformly RC-positive for $m\in\NN^+, k\in\NN^+$ and $1\leq p\leq \rk{E}$;
        \item every quotient bundle $Q$ of $E$ is uniformly RC-positive;
        \item every line sub-bundle $L$ of $E^*$ is not pseudo-effective.
    \end{enumerate}
\end{proposition}

For a line bundle, RC-positivity and uniform RC-positivity coincide.
Moreover, there are some equivalent characterizations of RC-positive line bundles.
\begin{theorem}[{\cite[Theorem 4.1]{Yan19}}] \label{Theorem:Characterization of RC-positive line bundle}
    Let $L$ be a line bundle over a compact complex manifold $X$ with $\dim X=n$.
    The following statements are equivalent:
    \begin{enumerate}
        \item the dual line bundle $L^*$ is not pseudo-effective;
        \item there exists a smooth Gauduchon metric $\omega_G$ on $X$ such that
        $$\int_{X}c_1^{\BC}(L)\cdot\omega_G^{n-1}>0,$$
        where $c_1^{\BC}(L)$ is the first Bott--Chern class of $L$.
        \item there exist a smooth Hermitian metric $h$ on $L$ and a Hermitian metric $\omega$ on $X$ such that
        the scalar curvature of the Chern curvature $R^h=-\sqrt{-1}\partial\bar\partial\log h$ with respect to $\omega$ is positive, i.e.
        $$s=\tr_\omega R^h>0;$$
        \item $L$ is RC-positive.
    \end{enumerate}
\end{theorem}

\begin{proposition}[{\cite[Corollary 3.1]{Yan20}}] \label{Proposition:Yang's vanishing}
    If $(E,h)$ is a uniformly RC-positive vector bundle over a compact complex manifold $X$,
    then
    $$H^0(X,(E^*)^{\otimes m})=0 \quad \text{for every } m\geq 1.$$
\end{proposition}

It is well known that on a rationally connected projective manifold $X$, one has
$$H^0(X,(T^*X)^{\otimes m})=0 \quad \text{for every } m\geq 1.$$
We collect some characterizations of rationally connected manifolds in terms of differential-geometric positivity notions.

\begin{theorem}[\cite{Yan20,LZZ25}] \label{Theorem:Characterization of rationally connected manifolds}
    Let $X$ be a projective manifold.
    \begin{enumerate} 
        \item  If the tangent bundle $TX$ is uniformly RC-positive,
        then $X$ is rationally connected.
        \item $X$ is rationally connected if and only if the line bundle $\mathcal{O}_{\Lambda^pT^*X}(-1)$ is uniformly RC-positive for every $1\leq p\leq \dim X$, where $\mathcal{O}_{\Lambda^pT^*X}(-1)$ is the dual bundle of the tautological line bundle $\mathcal{O}_{\Lambda^pT^*X}(1)$ of $\mathbb{P}(\Lambda^pT^*X)$, the projective bundle of hyperplanes of $\Lambda^pT^*X$.
        \item $X$ is rationally connected if and only if $TX$ is mean curvature positive, 
        that is, there exist a Hermitian metric $\omega$ on $X$ and a Hermitian metric $h$ on $TX$ such that
        $i\Lambda_\omega\Theta_{TX,h}>0$.
    \end{enumerate}
\end{theorem}

Li--Zhang--Zhang  establish an equivalence between mean curvature positivity and so-called HN-positivity. 

Assume $(X,\omega)$ is a compact Gauduchon manifold and $\mathcal{F}$ is a coherent sheaf over $X$. 
The $\omega$-degree of $\mathcal{F}$ is given by
$$\deg_\omega(\mathcal{F}):=\deg_\omega(\det\mathcal{F})=\int_{X}c_1^{\BC}(\det\mathcal{F})\wedge\frac{\omega^{n-1}}{(n-1)!}
=\frac{1}{2\pi}\int_{X}i\Theta_{\det\mathcal{F},h}\wedge\frac{\omega^{n-1}}{(n-1)!},$$
where $h$ is an arbitrary Hermitian metric on $\det\mathcal{F}$.
The $\omega$-slope of $\mathcal{F}$ is defined by 
$$\mu_\omega(\mathcal{F}):=\frac{\deg_\omega(\mathcal{F})}{\rk\mathcal{F}}.$$
Let $E$ be a holomorphic vector bundle over $X$.
Define
$$\mu_L(E,\omega):=\inf\{\mu_\omega(\mathcal{Q})~|~ \mathcal{Q} \text{ is a coherent quotient sheaf of } E, ~\rk\mathcal{Q}>0\}.$$
Then for a given Gauduchon metric $\omega$, $E$ is called $\omega$-$\HN$-positive if $\mu_L(E,\omega)>0$.
The bundle $E$ is called $\HN$-positive if there is a Gauduchon metric $\omega$ on $X$ such that $\mu_L(E,\omega)>0$.

\begin{theorem}[{\cite[Theorem~1.4]{LZZ25}}]\label{Theorem:Mean curvature positivity and HN-positivity}
    Let $(X,\widetilde{\omega})$ be a compact Hermitian manifold and $E$ a holomorphic vector bundle over $X$. 
    Then there exists a Hermitian metric $h$ such that the mean curvature $\sqrt{-1}\Lambda_{\widetilde{\omega}}\Theta_{E,h}$ is positive 
    if and only if $\mu_L(E,\omega)>0$ where $\omega$ is a Gauduchon metric conformal to $\widetilde{\omega}$.
\end{theorem}

\subsection{A counterexample}
In this subsection, we examine the RC-positivity of the Fermat quartic $K3$ surface.
This example shows why the uniform condition is essential in the rational connectedness theorem.
\begin{example}
    Let $$X:=\{[Z_0:Z_1:Z_2:Z_3]\in\mathbb{P}^3~|~Z_0^4+Z_1^4+Z_2^4+Z_3^4=0\}.$$
    For the metric induced by the Fubini--Study metric on $\mathbb{P}^3$, $TX$ is RC-positive.
    However, $X$ is not rationally connected, $-K_X$ is not RC-positive, and $TX$ admits no
    uniformly RC-positive Hermitian metric.
\end{example}
\begin{proof}
    It is easy to see that $X$ is a smooth projective surface and $TX$ is a holomorphic subbundle of $T\mathbb{P}^3$.
    Let $h_{\FS}$ be the Fubini--Study metric on $T\mathbb{P}^3$ and $h$ the induced metric on $TX$. 
    Note that $h_{\FS}$ is Griffiths positive, that is, at each point $q\in X$, and for any nonzero vectors $u,v\in T_q\mathbb{P}^3$,
    $$R^{h_{\FS}}(u,\bar{u},v,\bar{v})>0.$$
    
    Let $\beta$ be the second fundamental form of $TX$ in $T\mathbb{P}^3$. 
    By formula~\ref{Formula:curvature of subbundle}, at each point $q\in X$, and for any nonzero vectors $u,e\in T_qX$,
    $$R^h(u,\bar{u},e,\bar{e})=R^{h_{\FS}}(u,\bar{u},e,\bar{e})-|\beta(u)(e)|^2_{h_Q},$$
    where $Q=T\mathbb{P}^3/TX$ is the quotient bundle and $h_Q$ is the induced metric on $Q$.
    For any fixed $0\neq e\in T_qX$, $\beta$ induces a complex linear map 
    $$\beta_e:T_qX\longrightarrow Q_q, \quad \beta_e(u)=\beta(u)(e).$$
    This is a linear map from a two-dimensional complex vector space to an one-dimensional one.
    Thus $\dim\ker \beta_e\geq 1$. 
    Choosing $0\neq u\in\ker \beta_e$, one has
    $$R^h(u,\bar{u},e,\bar{e})=R^{h_{\FS}}(u,\bar{u},e,\bar{e})-|\beta(u)(e)|^2_{h_Q}>0.$$
    This shows that $TX$ is RC-positive.
    
    The adjunction formula implies that 
    $$K_X\simeq K_{\mathbb{P}^3}\otimes\mathcal{O}_{\mathbb{P}^3}(X)|_X\simeq\mathcal{O}_{\mathbb{P}^3}(-4)\otimes\mathcal{O}_{\mathbb{P}^3}(4)|_X\simeq\mathcal{O}_X.$$
    Therefore, $-K_X$ is not RC-positive, $TX$ is not uniformly RC-positive, and $X$ is not rationally connected by
    Propositions~\ref{Proposition:Basic property for uniformly RC-positive vector bundle}, \ref{Proposition:Yang's vanishing} and Theorems~\ref{Theorem:Characterization of RC-positive line bundle}, \ref{Theorem:Characterization of rationally connected manifolds}.
\end{proof}

\section{Uniform RC-positivity of rationally connected manifolds}\label{Sec:Equivalence}

In this section, we show that the tangent bundle of a rationally connected manifold is uniformly RC-positive.
An important property of rationally connected manifolds of dimension at least three is the existence of an  embedded very free rational curve.
Our basic idea is to construct a uniformly RC-positive metric by its local deformation family.

\subsection{The embedded very free curve}

Let us recall some basic properties of rationally connected manifolds (see \cite{AK03,Kol96}).

Let $X$ be a projective manifold. Consider the scheme $\Hom(\mathbb{P}^1,X)$.  
For any morphism $f:\mathbb{P}^1\rightarrow X$, write $[f]$ for the corresponding point of $\Hom(\mathbb{P}^1,X)$. 
By \cite[Theorem~5]{AK03}, one has
$$T_{[f]}\Hom(\mathbb{P}^1,X)=H^0(\mathbb{P}^1,f^*TX).$$
Moreover, if $H^1(\mathbb{P}^1,f^*TX)=0$, then $\Hom(\mathbb{P}^1,X)$ is smooth at $[f]$ and has dimension $N=h^0(\mathbb{P}^1,f^*TX)$.

Assume that $X$ is a projective rationally connected manifold of dimension $n\geq 3$. 
There exists an embedding $\iota:\mathbb{P}^1\rightarrow X$ passing through any given finite subset of $X$ such that $\iota^*TX$ is ample, see e.g. \cite[Chapter IV, Theorem~3.9]{Kol96}.
In other words, $$\iota^*TX\simeq\bigoplus_{j=1}^n\mathcal{O}_{\mathbb{P}^1}(a_j), \quad a_j\geq 1.$$
In particular, 
$$H^1(\mathbb{P}^1,\iota^*TX)=0.$$
Therefore, $\Hom(\mathbb{P}^1,X)$ is smooth at $[\iota]$.
Moreover, nearby deformations of a very free rational curve are still very free (see \cite[Section~3, statement after Remark~9]{AK03}).
Furthermore there is a small ball $B\subset \mathbb{C}^N$ centered at $[\iota]$ such that the evaluation morphism 
$F:\mathbb{P}^1\times B\rightarrow X$ is a submersion, see also \cite[Proposition~10]{AK03} (cf. \cite[Chapter II, Corollary~3.5.4]{Kol96}).

\subsection{Analytic construction of the uniformly RC-positive metric}

Before the construction of a uniformly RC-positive metric by the embedded very free curve $[\iota]$, we need the following Lemma~\ref{Lemma:Construct uniform RC-positive metric}.
\begin{lemma}\label{Lemma:Construct uniform RC-positive metric}
    Let $X$ be a connected compact K\"ahler manifold of dimension $n\geq 2$, and let $(E,h)$ be a holomorphic Hermitian vector bundle on $X$. 
    Let $B\subset \mathbb{C}^N$ be a ball and 
    $$F:\mathbb{P}^1\times B\rightarrow X, \quad F(t,b)=\iota_b(t)$$
    a holomorphic submersion.
    Let $\omega_{\FS}$ be the Fubini--Study metric on $\mathbb{P}^1$ with $\int_{\mathbb{P}^1}\omega_{\FS}=1$. 
    Assume there is a constant $\kappa>0$ such that 
    $$i\Theta_{\iota_b^*E,\iota_b^*h}\geq\kappa\omega_{\FS}\otimes\Id_{\iota_b^*E} \quad \text{for every } b\in B.$$
    Then there is a smooth real function $\varphi$ such that $(E,e^{-\varphi}h)$ is uniformly RC-positive.
\end{lemma}
We first give the following reformulation of the well-known result used in \cite[Section~4]{Yan19} (cf. \cite{Gau77} or \cite[Theorem~2.2]{CTW19}).

\begin{theorem}\label{Theorem:Solve equations}
    Let $X$ be a connected compact complex manifold of dimension $n\geq 2$. 
    Let $\omega_G$ be a Gauduchon metric, and let $dV$ be a smooth positive volume form.
    For $f\in C^\infty(X,\mathbb{R})$, the equations 
    $$\ddbar\phi\wedge\omega_G^{n-1}=f dV, \quad \int_{X}\phi dV=0$$
    have a unique smooth real solution $\phi$ if and only if $\int_{X}f dV=0$.
\end{theorem}

\begin{proof}[Proof of Lemma~\ref{Lemma:Construct uniform RC-positive metric}]
    Let $p_1:\mathbb{P}^1\times B\rightarrow\mathbb{P}^1$ and $p_2:\mathbb{P}^1\times B\rightarrow B$ denote the two projections.
    Choose a nonnegative smooth form $\mu\in C_c^\infty(B,\Lambda^{N,N})$ with $\int_B\mu=1$.
    Define 
    $$\nu:=F_*(p_1^*\omega_{\FS}\wedge p_2^*\mu)\in A^{n,n}(X) \quad \text{and} \quad \Omega_0:=F_*(p_2^*\mu)\in A^{n-1,n-1}(X).$$
    For a smooth nonnegative function $\psi$ and a smooth semipositive $(1,1)$-form $\alpha$, the definition of pushforward and Fubini's theorem imply that 
    $$\int_{X}\psi\alpha\wedge\Omega_0=\int_B\left(\int_{\mathbb{P}^1}(\psi\circ\iota_b)\iota_b^*\alpha \right)\mu(b)\geq 0.$$
    This shows that $\Omega_0\geq 0$. 
    Similarly, one can obtain $\nu\geq 0$. 
    By a direct computation
    $$\int_{X}\nu=\int_{\mathbb{P}^1\times B}p_1^*\omega_{\FS}\wedge p_2^*\mu=\left(\int_{\mathbb{P}^1}\omega_{\FS}\right)\left(\int_{B}\mu\right)=1.$$
    Note that $d\mu=0$ since $\mu$ is of top degree.
    Therefore, 
    $$d\Omega_0=F_*d(p_2^*\mu)=0.$$
    
    Fix a K\"ahler form $\omega$ on $X$ such that 
    $$\int_{X}dV=1 \quad \text{where } dV=\frac{\omega^n}{n!}.$$
    Write 
    $$\Omega_1=\frac{\omega^{n-1}}{(n-1)!} \quad \text{with } \omega\wedge\Omega_1=ndV.$$
    By compactness, there is a constant $C\geq1$ such that for any point $x\in X$, 
    $$R^h(u,\bar{u},v,\bar{v})\geq -C|u|^2_\omega|v|^2_h \quad \text{for every } u\in T_xX, v\in E_x.$$
    Take $\varepsilon\in(0,\frac{\kappa}{2nC})$ and define $\Omega_\varepsilon:=\Omega_0+\varepsilon\Omega_1$.
    Then $\Omega_\varepsilon>0$ and $d\Omega_\varepsilon=0$.
    In \cite[pp. 279--280]{Mic82}, Michelsohn observed that there is a Gauduchon metric $\omega_\varepsilon$ such that $\omega_\varepsilon^{n-1}=\Omega_\varepsilon$.
    Write $\nu=\rho dV$ for some smooth function $\rho$.
    Then by Theorem~\ref{Theorem:Solve equations}, there is a smooth real function $\phi_\varepsilon$ such that 
    \begin{equation}
        \ddbar\phi_\varepsilon\wedge\Omega_\varepsilon=\kappa(\rho-1)dV=\kappa(\nu-dV) \quad \text{and} \quad \int_X\phi_\varepsilon dV=0.
    \end{equation}    
    Define $\varphi=-\phi_\varepsilon$ and $H=e^{-\varphi}h=e^{\phi_\varepsilon}h$.
    We now show that $(E,H)$ is uniformly RC-positive.
    
    By the formula of conformal change, we have
    $$R^H=e^{\phi_\varepsilon}
    (R^h-\ddbar\phi_\varepsilon\otimes h).$$
    For any point $x\in X$ and any tangent vector $u\in T_xX$, define continuous functions
    \begin{align*}
        \lambda_h(x;u):=\inf_{0\neq v\in E_x}
        \frac{R^h(u,\bar{u},v,\bar{v})}{|v|_h^2} 
        \quad \text{and}\quad 
        \lambda_H(x;u):=\inf_{0\neq v\in E_x}
        \frac{R^H(u,\bar{u},v,\bar{v})}{|v|_H^2}.
    \end{align*}
    Then we have
    $$\lambda_H(x;u)
    =\lambda_h(x;u)-(\ddbar\phi_\varepsilon)(u,\bar{u}).$$
    Define a continuous function on $X$ by
    $$M_\varepsilon(x):=
    \sup_{\substack{u\in T_xX\\ |u|_\omega=1}}
    \lambda_H(x;u).$$
    By the definition, for every $u\in T_xX$, we have
    $$\lambda_H(x;u)\leq M_\varepsilon(x)|u|_\omega^2.$$
    Therefore,
    \begin{equation}
        (\ddbar\phi_\varepsilon)(u,\bar{u})
        \geq \lambda_h(x;u)-M_\varepsilon(x)|u|_\omega^2.
    \end{equation}
    
    We first estimate $\ddbar\phi_\varepsilon\wedge\Omega_0$.
    For any $b\in B$, $t\in\mathbb{P}^1$ and
    $\zeta\in T_t\mathbb{P}^1$, the assumption on
    $(\iota_b^*E,\iota_b^*h)$ implies that
    $$\lambda_h\bigl(\iota_b(t);d\iota_b(\zeta)\bigr)
    \geq\kappa\omega_{\FS}(\zeta,\bar{\zeta}).$$
    Taking $u=d\iota_b(\zeta)$, we obtain
    \begin{align*}
        (\iota_b^*(\ddbar\phi_\varepsilon))
        (\zeta,\bar{\zeta})
        \geq\kappa\omega_{\FS}(\zeta,\bar{\zeta})
        -M_\varepsilon(\iota_b(t))
        (\iota_b^*\omega)(\zeta,\bar{\zeta}).
    \end{align*}
    This means that, as an inequality of continuous real $(1,1)$-forms on
    $\mathbb{P}^1$,
    $$\iota_b^*(\ddbar\phi_\varepsilon)
    \geq\kappa\omega_{\FS}
    -(M_\varepsilon\circ\iota_b)\iota_b^*\omega.$$
    For any smooth nonnegative function $\psi$ on $X$, we have
    \begin{align*}
        \int_X\psi\,\ddbar\phi_\varepsilon\wedge\Omega_0
        &=\int_B\left(\int_{\mathbb{P}^1}(\psi\circ\iota_b)\iota_b^*(\ddbar\phi_\varepsilon)\right)\mu(b)           \\
        &\geq\kappa\int_B\left(\int_{\mathbb{P}^1}
        (\psi\circ\iota_b)\omega_{\FS}\right)\mu(b)
        -\int_B\left(\int_{\mathbb{P}^1}
        ((\psi M_\varepsilon)\circ\iota_b)\iota_b^*\omega\right)\mu(b)     \\
        &=\kappa\int_X\psi\nu-\int_X\psi M_\varepsilon\omega\wedge\Omega_0.
    \end{align*}
    This implies that
    \begin{equation}
        \ddbar\phi_\varepsilon\wedge\Omega_0\geq\kappa\nu-M_\varepsilon\,\omega\wedge\Omega_0.
    \end{equation}
    
    We now estimate $\ddbar\phi_\varepsilon\wedge\Omega_1$.
    By the choice of $C$, we have
    $$\lambda_h(x;u)\geq-C|u|_\omega^2.$$
    Therefore, for every $x\in X$ and $u\in T_xX$,
    $$ (\ddbar\phi_\varepsilon)(u,\bar{u})\geq-(C+M_\varepsilon(x))|u|_\omega^2.$$
    This gives an inequality of Hermitian forms
    $$\ddbar\phi_\varepsilon\geq-(C+M_\varepsilon)\omega.$$
    Since $\Omega_1$ is positive and $\omega\wedge\Omega_1=ndV$, we obtain
    \begin{equation}
        \ddbar\phi_\varepsilon\wedge\Omega_1\geq-n(C+M_\varepsilon)dV.
    \end{equation}
    
    Combining the two estimates and using
    $\Omega_\varepsilon=\Omega_0+\varepsilon\Omega_1$, we have
    \begin{align*}
        \ddbar\phi_\varepsilon\wedge\Omega_\varepsilon
        &\geq
        \kappa\nu-M_\varepsilon\,\omega\wedge\Omega_0
        -\varepsilon n(C+M_\varepsilon)dV\\
        &=\kappa\nu-\varepsilon nC\,dV
        -M_\varepsilon\,\omega\wedge\Omega_\varepsilon.
    \end{align*}
    On the other hand, the construction of $\phi_\varepsilon$ implies that
    $$\ddbar\phi_\varepsilon\wedge\Omega_\varepsilon
    =\kappa\nu-\kappa dV.$$
    Hence, we obtain
    $$M_\varepsilon\,\omega\wedge\Omega_\varepsilon
    \geq(\kappa-\varepsilon nC)dV
    \geq\frac{\kappa}{2}dV.$$
    Therefore, 
    $$M_\varepsilon(x)\geq\frac{\kappa}{2}\frac{dV}{\omega\wedge\Omega_\varepsilon}(x)>0
    \quad \text{for every } x\in X.$$
    At each point $x$, choose an $\omega$-unit tangent vector $u_x$ at
    which the supremum defining $M_\varepsilon(x)$ is attained.
    Then for every $0\neq v\in E_x$,
    $$R^H(u_x,\bar{u}_x,v,\bar{v})\geq M_\varepsilon(x)|v|_H^2>0.$$
    This shows that $(E,H=e^{-\varphi}h)$ is uniformly RC-positive.
    This completes the proof of Lemma~\ref{Lemma:Construct uniform RC-positive metric}.
    
\end{proof}

We now prove Theorem~\ref{Main Thm 0}.
\begin{theorem}[=Theorem~\ref{Main Thm 0}]\label{Main Thm 0'}
    Let $X$ be a compact K\"ahler manifold.
    If $X$ is rationally connected, then $TX$ is uniformly RC-positive.
\end{theorem}
\begin{proof}
    If $n:=\dim X=1$, $X\simeq\mathbb{P}^1$. 
    In this case, the standard Fubini--Study metric gives a Griffiths positive Hermitian metric on $TX$. 
    In particular, $TX$ is uniformly RC-positive.
    
    If $n=2$, $X$ is a rational surface.
    By the classification of minimal rational surfaces, there is a sequence of point blow-ups
    $$X=X_l\longrightarrow X_{l-1}\longrightarrow\cdots\longrightarrow X_0,$$
    where $X_0$ is either $\mathbb{P}^2$ or a Hirzebruch surface.
    Note that the Fubini--Study metric on $\mathbb{P}^2$ has positive holomorphic sectional curvature.
    By Hitchin's result \cite{Hit75}, every Hirzebruch surface also admits a K\"ahler metric with positive holomorphic sectional curvature.
    Hence $TX_0$ is uniformly RC-positive by
    \cite[Theorem~5.1]{Yan20}.
    Applying Theorem~\ref{Main Thm 1} (proved in Section~\ref{Sec:Blow-up}) to the point blow-ups above, we obtain that $TX$ is uniformly RC-positive.
    
   \
   
    We now assume that $n\geq3$.
    
    \medskip
    \noindent\textbf{Step 1: construct a metric via an embedded very free curve.}
    
    By the construction of the embedded very free curve above, there is an embedding
    $$\iota:\mathbb{P}^1\longrightarrow X$$
    such that
    $$\iota^*TX\simeq\bigoplus_{j=1}^n\mathcal{O}_{\mathbb{P}^1}(a_j),\quad a_j\geq1.$$
    Set $C:=\iota(\mathbb{P}^1)$.
    Let $\omega_{\FS}$ be the Fubini--Study metric on
    $\mathbb{P}^1$ with $\int_{\mathbb{P}^1}\omega_{\FS}=1.$
    The standard Fubini-Study metric induce a metric $h_C$ on $\iota^*TX$ through the holomorphic splitting such that 
    $$i\Theta_{\iota^*TX,h_C} \geq2\kappa\omega_{\FS}\otimes\Id_{\iota^*TX}$$
    for some constant $\kappa>0$.
    
    Since $\iota$ is an embedding, we can regard $h_C$ as a smooth
    Hermitian metric on $TX|_C$.
    Consider the smooth real vector bundle of Hermitian forms on $TX$.
    By \cite[Lemma~10.12]{Lee13}, $h_C$ extends to a global smooth section $\widehat{h}$ of this bundle, with
    $$\widehat{h}|_C=h_C.$$
    Since positive definiteness is an open condition, $\widehat{h}$
    is positive definite on an open neighborhood $U$ of $C$.
    Choose a smaller open neighborhood $U_1$ of $C$ such that $U_1\Subset U$.
    By \cite[Proposition~2.25]{Lee13}, there is a smooth function
    $\chi:X\rightarrow[0,1]$ satisfying
    $$\chi=1\text{ on }\overline{U_1}
    \quad\text{and}\quad
    \supp\chi\subset U.$$
    Let $k$ be a fixed global smooth Hermitian metric on $TX$, and set
    $$h:=\chi\widehat{h}+(1-\chi)k.$$
    Then $h$ is a global smooth Hermitian metric on $TX$ satisfying
    $$\iota^*h=h_C.$$
    Therefore,
    $$\iota^*(i\Theta_{TX,h}) =i\Theta_{\iota^*TX,\iota^*h} =i\Theta_{\iota^*TX,h_C}
    \geq2\kappa\omega_{\FS}\otimes\Id_{\iota^*TX}.$$
    
    \medskip
    \noindent\textbf{Step 2: obtain the uniform RC-positivity by Lemma~\ref{Lemma:Construct uniform RC-positive metric}.}
    
    By the construction of the embedded very free curve above, there exist a small ball $B\subset \mathbb{C}^N$ centered at $0$ and a submersion 
    $F:\mathbb{P}^1\times B\rightarrow X$.
    
    For every $b\in B$, write $\iota_b(t)=F(t,b)$ and $h_b:=\iota_b^*h$.
    These metrics are the restrictions of the same smooth metric $F^*h$ on $F^*TX$.
    Define a continuous function on $\mathbb{P}^1\times B$ by
    $$\mu(t,b):=
    \inf_{\substack{
            u\in T_t\mathbb{P}^1,\ |u|_{\omega_{\FS}}=1\\
            v\in(\iota_b^*TX)_t,\ |v|_{h_b}=1}}
    R^{h_b}_t(u,\bar{u},v,\bar{v}).$$
    By Step 1 and the equality $\iota^*h=h_C$, we have
    $$\mu(t,0)\geq2\kappa
    \quad\text{for every }t\in\mathbb{P}^1.$$
    For each $t\in\mathbb{P}^1$, choose a neighborhood
    $U_t\times V_t$ of $(t,0)$ such that
    $$\mu>\kappa\quad\text{on }U_t\times V_t.$$
    Since $\mathbb{P}^1$ is compact, one can find a smaller ball $B'\subset B$ centered at $0$, such that
    $$R^{h_b}_t(u,\bar{u},v,\bar{v})\geq\kappa|u|_{\omega_{\FS}}^2|v|_{h_b}^2$$
    for every $(t,b)\in\mathbb{P}^1\times B'$,
    $u\in T_t\mathbb{P}^1$ and
    $v\in(\iota_b^*TX)_t$.
    Equivalently,
    $$i\Theta_{\iota_b^*TX,\iota_b^*h}
    \geq\kappa\omega_{\FS}\otimes\Id_{\iota_b^*TX}
    \quad\text{for every }b\in B'.$$
    
    We have therefore verified all the assumptions of
    Lemma~\ref{Lemma:Construct uniform RC-positive metric} with $E=TX$ and the fixed metric $h$ constructed in Step 1.
    Hence $TX$ is uniformly RC-positive.

\end{proof}

\section{Stability under blow-ups} \label{Sec:Blow-up}

In this section, we show that uniform RC-positivity of the tangent bundle is preserved under blow-ups along smooth centers.
We first consider a natural variant of a squared distance function near the center of the blow-up and the behavior of its pull-back near the exceptional divisor.

\begin{lemma}\label{Lemma:auxiliary function 1}
    Let $Y$ be a connected compact complex submanifold of a complex manifold $X$, of dimension $m\geq 0$, and of codimension $r\geq 2$.
    There exist a neighborhood $U_0$ of $Y$, a function $\rho\in C^\infty(U_0,[0,\infty))$, a Hermitian metric $g$ on $TX|_{U_0}$, and constants
    $$c_0,C_P,C_A,\varepsilon>0$$
    with the following properties:
    \begin{enumerate}
        \item $\rho^{-1}(0)=Y$;
        \item on $0<\rho<\varepsilon$,
        \begin{equation} \label{Formula:Prho}
            P_\rho:=\frac{\partial\rho\otimes\bar\partial\rho}{\rho} \quad \text{satisfies} \quad 0\leq P_\rho\leq C_P\cdot g;
        \end{equation}
        \item with $A_\rho:=\ddbar\rho-P_\rho$,
        \begin{equation}\label{Formula:Arho}
             A_\rho\ge-C_A\sqrt\rho\,g;
        \end{equation}
        \item after shrinking $U_0$ and $\varepsilon$ if necessary, for each $q\in U_0\setminus Y$ with $\rho(q)<\varepsilon$, there is a $g$-unit vector $u\in T_qX$ such that
        \begin{equation}\label{Formula:partial rho}
            \partial\rho(u)=0 \quad \text{and} \quad (\ddbar\rho)(u,\bar{u})\geq c_0.
        \end{equation}
    \end{enumerate}
    All tensor inequalities are inequalities of Hermitian forms.
\end{lemma}

\begin{proof}
    Take finitely many relatively compact adapted holomorphic charts $U_\alpha$ with coordinates
    $$(y_\alpha,w_\alpha)\in\CC^m\times\CC^r, \quad Y\cap U_\alpha=\{w_\alpha=0\}.$$
    Take a neighborhood $U_1\Subset\bigcup U_\alpha$ of $Y$.
    By shrinking the neighborhood and using the partition of unity, one can find nonnegative smooth functions $\chi_\alpha$ satisfying
    $$\supp \chi_\alpha\Subset U_1\cap U_\alpha \quad \text{and} \quad \sum\chi_\alpha=1 \text{ on } U_0,$$
    where $U_0\Subset U_1$ is a smaller neighborhood of $Y$. 
    Set $$\rho:=\sum\chi_\alpha|w_\alpha|^2.$$
    Then $\rho$ is a smooth function on $U_0$ and $\rho^{-1}(0)=Y$.
    
    Fix one of a finite family of smaller relatively compact charts, $(U_{\alpha_1};(y_{\alpha_1},w_{\alpha_1}))$, with coordinates denoted by $(y,w)$.
    Then for every index $\alpha$ with corresponding terms whose supports meet this chart near $Y$, the holomorphic change of defining functions has the form
    \begin{equation}\label{Formula:w alpha change}
        w_\alpha=A_\alpha(y)w+O(|w|^2),
    \end{equation}
    where $A_\alpha(y)$ is invertible along $Y$, and
    $\chi_\alpha(y,w)=\chi_\alpha(y,0)+O(|w|)$.
    Taylor expansion, with $y$ as a parameter, gives
    \begin{equation}\label{Formula:rho-Taylor}
        \rho(y,w)=Q_y(w)+E(y,w),
        \quad \text{where} \quad
        Q_y(w):=\sum_\alpha\chi_\alpha(y,0)|A_\alpha(y)w|^2.
    \end{equation}
    Set $s:=|w|$. 
    On smaller charts, we have
    \begin{equation}\label{Formula:estimate Error term}
        E=O(s^3), \quad \partial E=O(s^2), \quad \text{and} \quad \ddbar E=O(s).
    \end{equation}
    Note that $\sum\chi_\alpha=1$ on $Y$ and $A_\alpha(y)$ is invertible along $Y$.
    One has 
    \begin{equation}\label{Formula:estimate rho}
        cs^2\leq\rho(y,w)\leq Cs^2
    \end{equation}
    for some constants $0<c<C$.
    Let $g$ be a smooth Hermitian metric on $TX|_{U_0}$.
    
    Denote $\ell_{y,w}:=\partial_wQ_y|_w$, regarded as a vector of $(1,0)$-forms on the
    normal coordinate space, and denote $\beta:=(0,\ell_{y,w})$.
    Formulas~(\ref{Formula:rho-Taylor}) and (\ref{Formula:estimate Error term}) imply that
    \begin{equation}\label{Formula:Hessian rho}
        \partial\rho=\beta+O(s^2),
        \qquad
        \ddbar\rho=
        \begin{pmatrix}
            O(s^2)&O(s)\\
            O(s)&\ddbar_wQ_y+O(s)
        \end{pmatrix}.
    \end{equation}
    In particular, $|\partial\rho|_{g^*}^2\le C\rho$, and therefore
    $$P_\rho(u,\bar u)=\frac{|\partial\rho(u)|^2}{\rho}\le C|u|_g^2.$$
    
    We now estimate $A_\rho$.
    By Cauchy--Schwarz inequality, for every normal vector $\zeta$, we have
    \begin{equation}\label{Formula:normal Hessian Q}
        (\ddbar_wQ_y)(\zeta,\bar\zeta)
        -\frac{|\ell_{y,w}(\zeta)|^2}{Q_y(w)}\ge0, \quad w\neq0.
    \end{equation}
    Put $\delta:=\partial\rho-\beta=O(s^2)$.
    Note that $\rho=O(s^2)$, $Q_y=O(s^2)$, and $\rho-Q_y=E=O(s^3)$.
    Therefore, 
    $$\left|\frac1\rho-\frac1{Q_y}\right|
    =\frac{|\rho-Q_y|}{\rho Q_y}=O(s^{-1}).$$
    By a direct computation, 
    \begin{align*}
        \frac{\partial\rho\otimes\bar\partial\rho}{\rho}
        -\frac{\beta\otimes\bar\beta}{Q_y}
        =\frac{\delta\otimes\bar\beta+\beta\otimes\bar\delta
        +\delta\otimes\bar\delta}{\rho}
        +\beta\otimes\bar\beta
        \left(\frac1\rho-\frac1{Q_y}\right).
    \end{align*}
    Since $|\beta|=O(s)$ and $|\delta|=O(s^2)$, every term on the right-hand side is $O(s)$.
    This means that 
    $$P_\rho=\frac{\beta\otimes\bar\beta}{Q_y}+R,$$
    where the error tensor $R$ satisfies $|R(u,\bar{u})|\leq Cs|u|^2_g$.
    By the definition, 
    \begin{equation*}
        A_\rho=\ddbar\rho-P_\rho=
        \begin{pmatrix}
            O(s^2)&O(s)\\
            O(s)&\ddbar_wQ_y+O(s)
        \end{pmatrix}
        -\frac{\beta\otimes\bar\beta}{Q_y}-R.
    \end{equation*}
    Note that $\beta=(0,\ell_{y,w})$. 
    Combining the estimate on $R$ and formula~(\ref{Formula:normal Hessian Q}), we obtain
    $$A_\rho\ge-Cs\,g\ge-C_A\sqrt\rho\,g.$$
    
    Finally, positive definiteness of $Q_y$ implies that
    $|\ell_{y,w}|_{g^*}\geq cs$.
    Thus
    $$\partial_w\rho=\ell_{y,w}+O(s^2)\neq0$$
    for $s>0$ sufficiently small.
    Since $\codim_XY=r\geq 2$, we have $$\dim\ker\partial_w\rho\geq r-1\geq 1.$$
    Here $\partial_w\rho$ is regarded as a linear function on the normal coordinate subspace.
    Then there is a nonzero unit normal vector $\zeta$ such that $\partial_w\rho(\zeta)=0$.
    By taking $u=(0,\zeta)$, we have 
    $$\partial\rho(u)=0.$$
    The normal-normal block
    $\ddbar_wQ_y+O(s)$ in formula (\ref{Formula:Hessian rho}) is uniformly positive definite, so, after one final shrinking, we obtain 
    $$(\ddbar\rho)(u,\bar u)\geq c_0.$$
    
    This completes the proof of Lemma~\ref{Lemma:auxiliary function 1}.
\end{proof}

We now consider the function on the blow-up. 
The pull-back of $\rho$ has an exact quadratic factor on the blow-up.

\begin{lemma}\label{Lemma:Levi form on exceptional divisor}
    Under the assumptions of Lemma~\ref{Lemma:auxiliary function 1}, let $\pi:\widetilde{X}:=\Bl_YX\longrightarrow X$ be the blow-up and set $\widetilde{\rho}:=\pi^*\rho$ on $\pi^{-1}(U_0)$.
    In a blow-up chart
    \[
    w=t\,v(\xi),
    \qquad v(\xi)=(1,\xi_2,\ldots,\xi_r),
    \]
    there is a smooth function $\Gamma$ on the chart, positive along
    $E=\{t=0\}$, such that
    \begin{equation}\label{Formula:factorization of wtrho}
        \widetilde\rho=|t|^2\Gamma(y,t,\xi),
        \quad \text{and} \quad
        \Gamma(y,0,\xi)=Q_y(v(\xi))>0.
    \end{equation}
    Consequently,
    \begin{equation}\label{Formula:Levi form on E}
        \ddbar\widetilde\rho|_{(y,0,\xi)}
        =Q_y(v(\xi))\,dt\otimes d\bar t
    \end{equation}
\end{lemma}
\begin{proof}
    The components $w_\alpha^1,\ldots,w_\alpha^r$ generate the ideal of $Y$ on $U_\alpha$.  
    In particular, every $w_\alpha^j$ is a local section of the ideal sheaf $\mathcal{I}_Y$.  
    In the chosen blow-up chart, the ideal $\mathcal{I}_Y\cdot\mathcal{O}_{\widetilde{X}}=\mathcal{O}_{\widetilde{X}}(-E)$ is generated by $t$. 
    Hence, on $\pi^{-1}(U_\alpha)$, one may write 
    $$\pi^*w_\alpha^j=t\,a_{\alpha j}(y,t,\xi).$$
    Then we have 
    $$\widetilde{\rho}=|t|^2\Gamma \quad \text{where} \quad \Gamma=\sum_\alpha\pi^*\chi_\alpha\sum_j|a_{\alpha j}|^2.$$
    By formula~(\ref{Formula:w alpha change}), we have
    $a_{\alpha}(y,0,\xi)=A_\alpha(y)v(\xi)$.
    Hence, $\Gamma(y,0,\xi)=Q_y(v(\xi))>0$.
    Differentiating $\widetilde{\rho}$ at $t=0$, we obtain
    $$\ddbar\widetilde\rho|_{(y,0,\xi)}
    =Q_y(v(\xi))\,dt\otimes d\bar t.$$
    
    This completes the proof of Lemma~\ref{Lemma:Levi form on exceptional divisor}.
\end{proof}

\begin{remark}
    The function $\rho$ constructed in Lemma~\ref{Lemma:auxiliary function 1} can be regarded as a generalization of the standard squared normal function $|w|^2$ when $Y$ is a single point.
    The conclusions of Lemmas~\ref{Lemma:auxiliary function 1} and \ref{Lemma:Levi form on exceptional divisor} are essentially estimates of Levi forms near $Y$ and on the exceptional divisor.
\end{remark}

We need the following constructions of auxiliary functions.
\begin{lemma}\label{Lemma:auxiliary function 2}
    Given constants $M,\eta\in (0,\infty)$, there exist constants $S>2$ and $L>0$ depending only on $M,\eta$ satisfying the following property.
    For any $a>0$ and any $A>M+1$, there is a smooth function $\lambda$ on $[0,\infty)$ such that
    \begin{enumerate}
        \item $\lambda\geq0$;
        \item $\lambda(\tau)=A\tau$ for $\tau\in[0,a]$;
        \item $\frac{\lambda(\tau)}{\tau}>M$ for $\tau\in(a,2a)$;
        \item $0\leq\frac{\lambda(\tau)}{\tau}\leq L$ and $\lambda'(\tau)\geq-\eta$ for $\tau\in[2a,Sa]$;
        \item $\lambda=0$ on a neighborhood of $[Sa,\infty)$.
    \end{enumerate}
\end{lemma}
\begin{proof}
    Choose $\delta_0\in(0,1/4)$ and $\theta_1\in C^\infty([0,\infty),[0,1])$ such that
    $$\theta_1\equiv 1 \text{ on } [0,1+\delta_0] \quad \text{and} \quad \theta_1\equiv 0 \text{ on } [2-\delta_0,\infty).$$
    
    For $\tau\in[0,2]$, put 
    $$\widehat{\lambda}_A(\tau):=\tau[(M+1)+(A-M-1)\theta_1(\tau)].$$
    Then $\widehat{\lambda}_A(\tau)=A\tau \text{ on } [0,1+\delta_0]$, $\widehat{\lambda}_A(\tau)=(M+1)\tau \text{ on } [2-\delta_0,2]$, and
    $$\frac{\widehat{\lambda}_A(\tau)}{\tau}\geq M+1>M \text{ for } \tau\in(1,2).$$
    Continue with $\widehat{\lambda}_A(\tau)=(M+1)\tau$ on $[2,3]$.
    Choose $\theta_2\in C^\infty([3,4],[0,M+1])$ such that
    $$\theta_2\equiv M+1 \text{ on } [3,3+\delta_0] \quad \text{and} \quad \theta_2\equiv 0 \text{ on } [4-\delta_0,4].$$
    Define 
    $$\widehat{\lambda}_A(\tau)=3(M+1)+\int_{3}^{\tau}\theta_2(t)dt \quad \text{for } 3\leq\tau\leq 4.$$
    Set $C=\widehat{\lambda}_A(4)\in[3(M+1),4(M+1)]$.
    It is easy to see that 
    $$\widehat{\lambda}_A(\tau)=C \quad \text{for } 4-\delta_0\leq\tau\leq 4.$$
    
    Choose a smooth nonincreasing function $\psi:\RR\rightarrow [0,1]$ such that
    $$\psi\equiv 1 \text{ on } (-\infty,\delta_0] \quad \text{and} \quad \psi\equiv 0 \text{ on } [1-\delta_0,\infty).$$
    Take 
    $$T\geq \frac{C\|\psi'\|_\infty}{\eta}$$
    and set 
    $$\widehat{\lambda}_A(\tau)=C\psi(\frac{\tau-4}{T}) \quad\text{for } \tau\geq4.$$
    Then we obtain a smooth function $\widehat{\lambda}_A(\tau)$ on $[0,\infty)$.
    
    Put $S:=4+T$ and $L:=2(M+1)$.
    For $\tau\geq 2$, we have 
    $$\widehat{\lambda}'_A(\tau)\geq-\eta \quad \text{and} \quad 0\leq \frac{\widehat{\lambda}_A(\tau)}{\tau} \leq L.$$
    Finally define
    $$\lambda(\tau)=a\widehat{\lambda}_A(\tau/a).$$
    Then $\lambda'(\tau)=\widehat{\lambda}'_A(\tau/a)$, and all assertions follow.
\end{proof}

We now prove Theorem~\ref{Main Thm 1}.
\begin{theorem}[=Theorem~\ref{Main Thm 1}]\label{Main Thm 1'}
    Let $X$ be a compact complex manifold and let $Y\subset X$ be a connected closed complex submanifold. 
    Let $\pi:\widetilde{X}:=\Bl_YX\longrightarrow X$ be the blow-up of $X$ with center $Y$.
    If $TX$ admits a smooth uniformly RC-positive Hermitian metric, then
    $T\widetilde{X}$ is also uniformly RC-positive.
\end{theorem}

\begin{proof}
    Let $h$ be a uniformly RC-positive Hermitian metric on $TX$. 
    We may assume that $\codim_XY\geq2$.
    
    \medskip
    \noindent\textbf{Step 1: basic setting and some constants.}
    
    Here we use the notation in Lemma~\ref{Lemma:auxiliary function 1} and Lemma~\ref{Lemma:auxiliary function 2}. 
    Choose a smaller neighborhood $U\Subset U_0$ of $Y$, and define $b:=\min_{\partial U}\rho>0$.
    Choose $0<\varepsilon_0<\min\{\varepsilon,b\}$, then
    $$\{x\in U~|~ \rho(x)\leq\varepsilon_0\}\Subset U.$$
    By applying Proposition~\ref{Proposition:Constants in RC-positive}, we know that
    $$C_1:=\inf_{q\in \overline{U}}\sup_{|u|_g=1}\inf_{|v|_h=1}R^h(u,\bar{u},v,\bar{v})>0,$$
    and there is a constant $0<C_2<\infty$ such that for any point $x\in \overline{U}$, every $g$-unit vector $u\in T_xX$ and every $v\in T_xX$,
    $$R^h(u,\bar{u},v,\bar{v})\geq-C_2|v|^2_h.$$
    
    Choose 
    $$M>\frac{C_2+1}{c_0} \quad \text{and} \quad 0<\eta<\frac{C_1}{4C_P}$$
    in Lemma~\ref{Lemma:auxiliary function 2}.
    Then we obtain two constants $S$ and $L$ depending only on $M$ and $\eta$.
    Choose $a>0$ in Lemma~\ref{Lemma:auxiliary function 2} sufficiently small that 
    $$Sa<\varepsilon_0 \quad \text{and} \quad C_AL\sqrt{Sa}<\frac{C_1}{4}.$$
    
    \medskip
    \noindent\textbf{Step 2: a smooth metric on the blow-up.}
    
    Set $\widetilde{\rho}=\pi^*\rho$ on $\pi^{-1}(U)$.
    Define a globally smooth semipositive tensor on $T\widetilde{X}$ by
    \begin{equation*}
        \widehat{h}(u,\bar{v}):= h(d\pi(u),\overline{d\pi(v)}).
    \end{equation*}
    Note that the tensor $\widehat{h}$ is positive definite on $\widetilde{X}\setminus E$.
    Let $k$ be a smooth Hermitian metric on $T\widetilde{X}$.
    Take a cut-off function 
    $\chi\in C^\infty([0,\infty),[0,1])$ satisfying
    $$\chi(s)=1\quad(s\le a/3) \quad \text{and} \quad \chi(s)=0\quad(s\ge2a/3).$$
    Then the function $\chi(\widetilde\rho)$ on $\pi^{-1}(U)$ vanishes on a neighborhood of its boundary, 
    since $\{x\in U~|~ \rho(x)\leq\varepsilon_0\}\Subset U$ and $Sa<\varepsilon_0$.
    Extend it by zero to a global smooth function
    $\widetilde\chi$ on $\widetilde{X}$, and set
    $$h_0:=\widetilde\chi k+(1-\widetilde\chi)\widehat{h}.$$
    Then $h_0$ is a smooth metric on $T\widetilde{X}$ and 
    $$h_0=\widehat{h} \quad \text{on} \quad\widetilde{X}\setminus\{\widetilde x\in\pi^{-1}(U)~|~ \widetilde\rho(\widetilde x)<2a/3\}$$
    Consider the compact subset
    $$K_a:=\pi^{-1}\bigl(\{x\in U~|~\rho(x)\le a\}\bigr).$$
    At points of $E$, Lemma~\ref{Lemma:Levi form on exceptional divisor} gives a positive normal eigenvalue of $\ddbar\widetilde\rho$.  
    At points of $K_a\setminus E$,
    Lemma~\ref{Lemma:auxiliary function 1}(4) gives a positive eigenvalue since $a<\varepsilon_0<\varepsilon$.  
    Thus the continuous function
    $$\widetilde x\longmapsto\sup_{|u|_{h_0}=1} (\ddbar\widetilde\rho)_{\widetilde x}(u,\bar u)$$
    is positive on $K_a$.
    Therefore
    $$c_a:=\inf_{\widetilde x\in K_a}\sup_{|u|_{h_0}=1}(\ddbar\widetilde\rho)_{\widetilde x}(u,\bar{u})>0.
    $$
    Compactness of $K_a$ also gives a constant $C_a<\infty$ such that
    $$R^{h_0}(u,\bar u,v,\bar{v})\geq-C_a|u|_{h_0}^2|v|_{h_0}^2 \quad \text{on }K_a.$$
    
    Choose a constant
    $$A>\max\{M+1,\frac{C_a+1}{c_a}\}.$$
    Then we can obtain a smooth function $\lambda$ by Lemma~\ref{Lemma:auxiliary function 2}.
    Define a smooth function
    $$f(\tau):=\int_{0}^{\tau}\frac{\lambda(s)}{s}ds.$$
    Note that $\lambda=0$ on a neighborhood of $[Sa,\infty)$. 
    Thus $f$ is constant there.
    This implies that 
    $f(\widetilde{\rho})=f(Sa)$ is constant on a neighborhood of the boundary of $\pi^{-1}(U)$ since $Sa<\varepsilon_0$.
    Then we define
    \begin{equation*}
        F:=
        \begin{cases}
            f(\widetilde\rho),&\text{on }\pi^{-1}(U),\\
            f(Sa),&\text{on }\widetilde{X}\setminus\pi^{-1}(U).
        \end{cases}
    \end{equation*}
    The argument above shows that $F$ is a globally defined smooth function on $\widetilde{X}$.
    Define a new Hermitian metric on $T\widetilde{X}$ by 
    $$H=e^{-F}h_0.$$
    We now show that $H$ is uniformly RC-positive.

    \medskip
    \noindent\textbf{Step 3: the curvature computation.}
    
    According to the formula of conformal change, 
    $$R^H=e^{-F}(R^{h_0}+\ddbar F\otimes h_0).$$
    
    \medskip
    \noindent\textbf{Region I: $\{\widetilde{\rho}\leq a\}\subset \pi^{-1}(U)$.}
    
    According to Lemma~\ref{Lemma:auxiliary function 2}(2), $\lambda(\widetilde{\rho})=A\widetilde{\rho}$ here.
    Thus 
    $$F=f(\widetilde{\rho})=A\widetilde{\rho} \quad \text{and} \quad \ddbar F=A\ddbar \widetilde{\rho}.$$
    At each point $\widetilde{x}$, choose an $h_0$-unit vector $\widetilde{u}$ satisfying
    $(\ddbar\widetilde\rho)(\widetilde{u},\overline{\widetilde{u}})\ge c_a$.  
    Then for every $0\neq \widetilde{v}\in T_{\widetilde{x}}\widetilde{X}$,
    \begin{align*}
        R^H(\widetilde{u},\overline{\widetilde{u}},\widetilde{v},\overline{\widetilde{v}}) \geq e^{-F}(-C_a+Ac_a)|\widetilde{v}|_{h_0}^2>0.
    \end{align*}
    This means that $H$ is uniformly RC-positive in $\{\widetilde{\rho}\leq a\}$.
    
    \medskip
    \noindent\textbf{Region II: $\{a< \widetilde{\rho}< 2a\}\subset \pi^{-1}(U)$.}
    
    We first compute $\ddbar F$. 
    By the definition, we have 
    $$f'(\widetilde{\rho})=\frac{\lambda(\widetilde{\rho})}{\widetilde{\rho}},
    \quad
    \partial F= f'(\widetilde{\rho})\partial\widetilde{\rho}, 
    \quad
    \text{and} \quad
    \bar\partial F=f'(\widetilde{\rho})\bar\partial\widetilde{\rho}.$$
    Then we obtain
    $$\ddbar F=f'(\widetilde{\rho})\ddbar\widetilde{\rho}+f''(\widetilde{\rho})\partial\widetilde{\rho}\otimes\bar\partial\widetilde{\rho}.$$
    By a direct computation, 
    $$f''(\widetilde{\rho})=\frac{\lambda'(\widetilde{\rho})}{\widetilde{\rho}}-\frac{\lambda(\widetilde{\rho})}{\widetilde{\rho}^2}.$$
    Therefore, 
    $$\ddbar F=\frac{\lambda(\widetilde{\rho})}{\widetilde{\rho}}(\ddbar\widetilde{\rho}
    -\frac{\partial\widetilde{\rho}\otimes\bar\partial\widetilde{\rho}}{\widetilde{\rho}})
    +\frac{\lambda'(\widetilde{\rho})}{\widetilde{\rho}}\partial\widetilde{\rho}\otimes\bar\partial\widetilde{\rho}.$$
    
    Let $\widetilde{x}$ be a point in $\{a< \widetilde{\rho}< 2a\}$ and $x=\pi(\widetilde{x})$.
    Note that $d\pi:T_{\widetilde x}\widetilde{X}\to T_xX$ is an isomorphism and $h_0=\widehat{h}$.
    By Lemma~\ref{Lemma:auxiliary function 1}(4), there is a $g$-unit tangent vector $u\in T_xX$ such that 
    $$\partial\rho(u)=0 \quad \text{and} \quad (\ddbar\rho)(u,\bar u)\geq c_0.$$
    According to Lemma~\ref{Lemma:auxiliary function 2}(3), $\lambda(\widetilde{\rho})/\widetilde{\rho}>M$ here.
    Thus 
    $$\ddbar F(\widetilde{u},\overline{\widetilde{u}})
    =\frac{\lambda(\widetilde{\rho})}{\widetilde{\rho}}(\ddbar\rho)(u,\bar{u})\geq Mc_0,$$
    where $\widetilde{u}=d\pi^{-1}(u)$.
    By the construction, for any $0\neq\widetilde{v}\in T_{\widetilde{x}}\widetilde{X}$ and $v=d\pi(\widetilde{v})$,
    $$ R^{h_0}(\widetilde{u},\overline{\widetilde{u}},\widetilde{v},\overline{\widetilde{v}})
    =R^h(u,\bar{u},v,\bar{v}) 
    \quad \text{and} \quad
    |\widetilde{v}|_{h_0}=|v|_h.$$
    Therefore, 
    $$R^{h_0}(\widetilde{u},\overline{\widetilde{u}},\widetilde{v},\overline{\widetilde{v}})
    +\ddbar F(\widetilde{u},\overline{\widetilde{u}})|\widetilde{v}|^2_{h_0}\geq(-C_2+Mc_0)|v|^2_h>0.$$
    This means that $H$ is uniformly RC-positive in $\{a<\widetilde{\rho}< 2a\}$.

    \medskip
    \noindent\textbf{Region III: $\{2a\leq\widetilde{\rho}\leq Sa\}\subset \pi^{-1}(U)$.}
    
    Again identify the tangent spaces by $d\pi$.
    According to the estimates of $P_\rho$ and $A_\rho$ in Lemma~\ref{Lemma:auxiliary function 1}, the construction of $M,\eta,a$, Lemma~\ref{Lemma:auxiliary function 2}(4),  and computation above,
    $$\ddbar F\geq -(C_AL\sqrt{\widetilde\rho}+C_P\eta)g\geq
    -\frac{C_1}{2}g.$$
    At each point $\widetilde{x}$, choose a $g$-unit tangent vector $u\in T_{\pi(\widetilde{x})}X$ such that 
    $$\inf_{|v|_h=1}R^h(u,\bar{u},v,\bar{v})\geq C_1>0.$$
    Then for $\widetilde{u}=d\pi^{-1}(u)$ and any $0\neq\widetilde{v}\in T_{\widetilde{x}}\widetilde{X}$, we have 
    $$R^{h_0}(\widetilde{u},\overline{\widetilde{u}},\widetilde{v},\overline{\widetilde{v}})
    +\ddbar F(\widetilde{u},\overline{\widetilde{u}})|\widetilde{v}|^2_{h_0}\geq\frac{C_1}{2}|v|^2_h>0,$$
    where $v=d\pi(\widetilde{v})$.
    This means that $H$ is uniformly RC-positive in $\{2a\leq\widetilde{\rho}\leq Sa\}$.

    \medskip
    \noindent\textbf{Region IV: $\widetilde{X}\setminus\{\widetilde{x}\in\pi^{-1}(U)~|~\widetilde{\rho}(\widetilde{x})\leq Sa\}$.}
    
    Note that $F$ is constant and $h_0=\widehat{h}$ and $\pi$ is biholomorphic here.
    Hence $H$ is a positive constant multiple of the original metric $h$, which is uniformly RC-positive.
    This means that $H$ is uniformly RC-positive in $\widetilde{X}\setminus\{\widetilde{x}\in\pi^{-1}(U)~|~\widetilde{\rho}(\widetilde{x})\leq Sa\}$.
    
    In summary, we construct a uniformly RC-positive smooth Hermitian metric $H$ on $T\widetilde{X}$.
    This completes the proof of Theorem~\ref{Main Thm 1}.
    
\end{proof}

\section{Hopf surfaces and surfaces of class VII} \label{Sec:Classification 1}

By \cite[formula (6.4)]{LY17}, tangent bundles of the standard $n$-dimensional Hopf manifolds $\mathbb{S}^1\times\mathbb{S}^{2n-1}$ are uniformly RC-positive.
In \cite[Proposition 5.8]{Yan17}, Yang showed that on every Hopf surface $H_{a,b}$ of class 1, there exists a Gauduchon metric with semi-positive holomorphic bisectional curvature.
 
Following the computation there, we prove that the tangent bundle of every Hopf surface in the generalized sense is uniformly RC-positive in this section.
Then we give the classification of compact complex surfaces whose tangent bundle is uniformly RC-positive, except minimal class VII surfaces with $b_2>0$.

A compact complex surface $X$ is called a Hopf surface if its universal covering is analytically isomorphic to $\CC^2\setminus\{0\}$.
It has been proved by Kodaira (see \cite{BPV84,GO98}) that its fundamental group $\pi_1(X)$ is a finite extension of an infinite cyclic group generated by a biholomorphic contraction whose inverse takes the form
$$(z,w)\longmapsto (az+\lambda w^m,bw)$$
where $a,b,\lambda\in\CC$, $|a|\geq|b|>1$, $m\in \NN^*$ and $\lambda(a-b^m)=0$.
Hence, there are two different cases of primary Hopf surfaces:
\begin{enumerate}
    \item The Hopf surface $H_{a,b}$ of class 1 if $\lambda=0$;
    \item The Hopf surface $H_{a,b,\lambda,m}$ of class 0 if $\lambda\neq 0$ and $a=b^m$.
\end{enumerate}
Any Hopf surface admits a finite, unramified covering which is a primary Hopf surface.

\subsection{Primary Hopf surfaces}
In this subsection, we prove the uniform RC-positivity of primary Hopf surfaces.

First consider the Hopf surface of class 1.
Actually, the computation in \cite{Yan17} has shown that the Gauduchon metric constructed there is uniformly RC-positive. 
We extract the proof here for the reader's convenience.
Let $H_{a,b}=\CC^2\setminus\{0\}/\sim$ where $(z,w)\sim (az,bw)$ and $|a|\geq|b|>1$.
Set 
$$k_1=\log|a|, \quad k_2=\log|b|, \quad \text{and} \quad \alpha=\frac{2k_1}{k_1+k_2}\in[1,2).$$
Let $\theta(z,w)$ be a real smooth function defined by (see \cite{GO98})
$$|z|^2e^{-\frac{k_1\theta}{\pi}}+|w|^2e^{-\frac{k_2\theta}{\pi}}=1.$$
Define a real smooth function
$$\Phi(z,w)=e^{\frac{k_1+k_2}{2\pi}\theta}.$$
Then 
$$|z|^2\Phi^{-\alpha}+|w|^2\Phi^{\alpha-2}=1$$
and 
$$\theta(az,bw)=\theta(z,w)+2\pi,\quad \text{and} \quad \Phi(az,bw)=|a||b|\Phi(z,w).$$
Consider the well-defined Gauduchon metric on $H_{a,b}$ given in \cite[Proposition 5.8]{Yan17} by
$$\omega:=\sqrt{-1}\left(\frac{\Phi^{-\alpha}}{\alpha^2}dz\wedge d\bar{z}+
\frac{\Phi^{\alpha-2}}{(2-\alpha)^2}dw\wedge d\bar{w}\right).$$
Write $z^1=z$ and $z^2=w$. 
Let $\omega=\sqrt{-1}\sum_{i,j=1}^{2}g_{i\bar{j}}dz^i\wedge d\bar{z}^j$ with $g_{i\bar{j}}=f_i\delta_{ij}$, where
$$f_1=\frac{\Phi^{-\alpha}}{\alpha^2} \quad \text{and} \quad f_2=\frac{\Phi^{\alpha-2}}{(2-\alpha)^2}.$$
Set 
$$A_{i\bar{j}}=\frac{\partial^2\log\Phi}{\partial z^i\partial\bar{z}^j} \quad \text{and} \quad 
\Delta=\alpha|z|^2\Phi^{-\alpha}+(2-\alpha)|w|^2\Phi^{\alpha-2}>0.$$
Then by \cite[Lemma 5.7]{Yan17}, the matrix
\begin{equation*}
    A=\frac{\Phi^{-2}}{\Delta^3}
    \begin{pmatrix}
        (2-\alpha)^2|w|^2 & \alpha(\alpha-2)\bar{z}w \\
        \alpha(\alpha-2)z\bar{w} & \alpha^2|z|^2
    \end{pmatrix}
\end{equation*}
is semi-positive and $\det A=0$.
Moreover, on $\CC^2\setminus\{0\}$, 
$$\tr A=\frac{\Phi^{-2}}{\Delta^3}((2-\alpha)^2|w|^2+\alpha^2|z|^2)>0.$$
Thus, $\rk A=1$ everywhere.
At every point $x\in H_{a,b}$, there exists $0\neq u\in T_xH_{a,b}$ satisfying 
$$A(u,\bar{u})>0.$$
By \cite[formula (5.17)]{Yan17}, $R_{i\bar{j}k\bar{l}}=A_{i\bar{j}}B_{k\bar{l}}$, where
\begin{equation*}
    B=\begin{pmatrix}
        \frac{\Phi^{-\alpha}}{\alpha} & 0 \\
        0 & \frac{\Phi^{\alpha-2}}{2-\alpha}
    \end{pmatrix}
\end{equation*}
is positive definite.
Therefore, for any $0\neq v\in T_xH_{a,b}$,
$$R^\omega(u,\bar{u},v,\bar{v})=A(u,\bar{u})B(v,\bar{v})>0.$$
This means that $\omega$ is uniformly RC-positive.
Combining the result of Yang, we obtain the following proposition.
\begin{proposition}[{cf.~\cite[Proposition 5.8]{Yan17}}]\label{Proposition:Hopf class 1}
    On every Hopf surface $H_{a,b}$ of class 1, there exists a Gauduchon metric $\omega$ with semi-positive holomorphic bisectional curvature.
    Moreover, $\omega$ is uniformly RC-positive.
\end{proposition}

Now consider the Hopf surface of class 0.
We obtain the following result.
\begin{proposition}\label{Proposition:Hopf class 0}
    On every Hopf surface $H_{a,b,\lambda,m}$ of class 0, there exists a uniformly RC-positive metric.
\end{proposition}
\begin{remark}
    We give a universal proof of the uniform RC-positivity of all general Hopf surfaces via the deformation and the positivity of Hopf surfaces of class 1 in Subsection~\ref{Subsection:Secondary Hopf surfaces}.
    We also provide an alternative proof here for primary Hopf surface of class 0 by an explicit deformation, which can be regarded as an explanation of our method in Subsection~\ref{Subsection:Secondary Hopf surfaces}.
\end{remark} 
\begin{proof}
    Let $$a=b^m, \quad F_t(z,w)=(az+tw^m,bw), \quad \text{and} \quad X_t:=H_{a,b,t,m}=(\CC^2\setminus\{0\})/\langle F_t \rangle.$$
    Let $\theta$ be the same proper smooth function constructed before associated to $X_0=H_{a,b}$.
    Set $$\tau=\theta/2\pi.$$
    Then we have $\tau\circ F_0=\tau+1$.
    Define a smooth map
    $$\Psi_t:\CC^2\setminus\{0\}\longrightarrow\CC^2\setminus\{0\}, \quad \Psi_t(z,w)=(z+\frac{t}{a}w^m\tau(z,w),w).$$
    Then 
    $$\Psi_t(F_0(z,w))=(az+\frac{t}{a}(bw)^m(\tau+1),bw)
    =(az+tw^m\tau+tw^m,bw),$$
    and 
    $$F_t(\Psi_t(z,w))=(a(z+\frac{t}{a}w^m\tau)+tw^m,bw)
    =(az+tw^m\tau+tw^m,bw).$$
    This means that 
    \begin{equation}\label{Equation:equi-invariant relation of Ft}
        \Psi_t\circ F_0=F_t\circ\Psi_t.
    \end{equation}
    This descends to a smooth map 
    $$\overline{\Psi_t}:X_0\longrightarrow X_t, \quad 
    \overline{\Psi_t}([p]_{F_0})=[\Psi_t(p)]_{F_t}.$$
    Take a compact subset 
    $$K:=\{0\leq\tau(z,w)\leq1\}\subset\CC^2\setminus\{0\}.$$
    It is easy to see 
    $$\Psi_t\rightarrow\Id \quad \text{in } C^\infty(K).$$
    This implies that for $|t|\ll1$, $d\Psi_t$ is invertible at every point in $K$.
    Note that 
    \begin{equation}\label{Equation:Differential of Ft}
        d\Psi_t|_{F_0(x)}\circ dF_0|_x=dF_t|_{\Psi_t(x)}\circ d\Psi_t|_x,
    \end{equation}
    where $dF_0|_x$ and $dF_t|_{\Psi_t(x)}$ are linear isomorphisms by the definition.
    Therefore, for any integer $k$, $d\Psi_t|_{F^k_0(x)}$ is invertible if and only if $d\Psi_t|_x$ is invertible.
    On the other hand, by $\tau\circ F_0=\tau+1$, we know that for any $x\in\CC^2\setminus\{0\}$, there is an integer $k=-\lfloor \tau(x) \rfloor$ such that $F^k_0(x)\in K$.
    This implies that $d\Psi_t$ is invertible at every point in $\CC^2\setminus\{0\}$.
    Hence, $\overline{\Psi_t}$ is a local diffeomorphism. 
    Therefore, the proper map $\overline{\Psi_t}$ is a finite covering. 
    Formula~(\ref{Equation:equi-invariant relation of Ft}) shows that $\overline{\Psi_t}$ satisfies
    $$(\overline{\Psi_t})_*:\pi_1(X_0)=\mathbb{Z}\longrightarrow\pi_1(X_t)=\mathbb{Z}, \quad 1\longmapsto 1,$$
    that is, it sends the generator $F_0$ to the generator $F_t$.
    Hence the covering $\overline{\Psi_t}$ has one sheet, and therefore $\overline{\Psi_t}$ is a diffeomorphism.
    
    Let $J_{X_t}$ be the complex structure on $X_t$ induced by the standard complex structure $J_{st}$ on $\CC^2\setminus\{0\}$. 
    Then we obtain a family of complex structures on $X_0$ by 
    $$J_t:=\overline{\Psi_t}^*J_{X_t}.$$
    On the covering space this is represented by
    $$\widetilde{J_t}=(d\Psi_t)^{-1}J_{st}d\Psi_t,$$
    since formula (\ref{Equation:Differential of Ft}) shows that $\widetilde{J_t}$ is $F_0$-invariant, which means that $\widetilde{J_t}$ descends to $J_t$.
    It is easy to see $\widetilde{J_t}\rightarrow J_{st}$ in $C^\infty(K)$.
    Then we obtain $J_t\rightarrow J_0$ in $C^\infty$ on the compact quotient.
    
    Let $g_0$ be the Riemannian metric associated with the Gauduchon metric $\omega$ constructed in Proposition~\ref{Proposition:Hopf class 1}.
    Define
    $$g_t(\xi,\eta):=\frac{1}{2}(g_0(\xi,\eta)+g_0(J_t\xi,J_t\eta)).$$
    Then $g_t$ is a $J_t$-Hermitian metric, and 
    $$g_t\rightarrow g_0 \quad \text{in } C^\infty.$$
    Therefore, the curvature tensors satisfy
    $$R^{g_t}\rightarrow R^{g_0} \quad \text{in } C^\infty.$$
    Define 
    $$\mu_t:=\inf_{x\in X_0}\sup_{|u|_{g_t}=1}\inf_{|v|_{g_t}=1}R^{g_t}(u,\bar{u},v,\bar{v}).$$
    By Proposition~\ref{Proposition:Hopf class 1}, $\mu_0>0$.
    Hence, by the continuous dependence of $\mu_t$ on $t$, 
    $$\mu_t>0 \quad \text{for } 0<|t|\ll 1.$$
    Since $g_t$ can be regarded as a Hermitian metric on $X_t$ via the diffeomorphism $\overline{\Psi_t}$, we obtain that $(X_t,g_t)$ is uniformly RC-positive for $t$ small enough.
    
    Finally, define a biholomorphic map
    $$G_s(z,w)=(sz,w), \quad s\neq 0.$$
    Then for any $\lambda\neq0$, we have 
    $$G_s\circ F_\lambda=F_{s\lambda}\circ G_s.$$
    This gives the isomorphism
    $$X_\lambda\simeq X_{s\lambda}.$$
    By taking $0\neq s$ such that $s\lambda$ small enough, we know that $X_\lambda$ admits a uniformly RC-positive metric.
    
\end{proof}

\subsection{Quotients of primary Hopf surfaces} \label{Subsection:Secondary Hopf surfaces}

In this subsection, we prove the uniform RC-positivity of general Hopf surfaces.

The first lemma uses precisely the special feature of Yang's formula.
\begin{lemma}\label{Lemma:Finite quotient of Yang's metric}
    Let $p:X\rightarrow Y=X/G$ be a finite \'etale Galois covering of compact complex surfaces.
    Suppose that $TX$ admits a Hermitian metric $h$ whose curvature factors as
    $$R^h=A\otimes B$$
    where $A$ is a smooth semi-positive $(1,1)$-form of rank one everywhere and $B$ is a positive definite Hermitian form on $TX$.
    Then $TY$ is uniformly RC-positive.
\end{lemma}
\begin{proof}
    
    \medskip
    \noindent\textbf{Step 1: construct a suitable metric on $X$.}
    
    For $\sigma\in G$, define $h_\sigma=\sigma^*h$, $A_\sigma=\sigma^*A$ and $B_\sigma=\sigma^*B$. 
    Then we have 
    $$R^{h_\sigma}=A_\sigma\otimes B_\sigma.$$
    Fix some point $x\in X$. 
    For each $\sigma$, the kernel $K_{\sigma,x}:=(\ker A_\sigma)_x$ is a proper complex line in the two-dimensional space $T_xX$,
    since $A_\sigma$ is semi-positive of rank one.
    A finite union of proper complex lines cannot cover the two-dimensional space $T_xX$.
    Hence, we can choose $0\neq u\in T_xX\setminus\bigcup_{\sigma}K_{\sigma,x}$.
    Then for every $\sigma\in G$ and every $0\neq v\in T_xX$, 
    $$R^{h_\sigma}(u,\bar{u},v,\bar{v})=A_\sigma(u,\bar{u})B_\sigma(v,\bar{v})>0.$$
    
    Consider the holomorphic Hermitian vector bundle
    $$(E,H):=\bigoplus_{\sigma\in G}(TX,h_\sigma)$$
    and the surjective holomorphic bundle map
    $$\Sigma:E\longrightarrow TX, \quad \Sigma((v_\sigma)_{\sigma\in G}):=\sum_{\sigma\in G}v_\sigma.$$
    Let $\widetilde{h}$ denote the quotient metric on $TX$ induced by $(E,H)$.
    It is easy to see that for the direction $u$ selected above and every $v_\sigma\in(TX,h_\sigma)$ with $0\neq (v_\sigma)_{\sigma\in G}$, 
    $$R^H(u,\bar{u},(v_\sigma)_{\sigma\in G},\overline{(v_\sigma)_{\sigma\in G}})
    =\sum_{\sigma\in G}R^{h_\sigma}(u,\bar{u},v_\sigma,\bar{v_\sigma})>0.$$
    This shows that $(E,H)$ is uniformly RC-positive.
    Hence $(TX,\widetilde{h})$ is also uniformly RC-positive by Proposition~\ref{Proposition:Basic property for uniformly RC-positive vector bundle} (or formula~(\ref{Formula:curvature of subbundle})).
    
    \medskip
    \noindent\textbf{Step 2: prove that $\widetilde{h}$ is $G$-invariant.}
    
    For $k\in G$, define a holomorphic linear map from $E$ to $k^*E$ by
    $$T_k:E_x\longrightarrow E_{kx}, \quad 
    (T_k(v))_\sigma:=dk_x(v_{\sigma k}),$$
    where $v=(v_\sigma)_{\sigma\in G}\in E_x$.
    Note that 
    \begin{align*}
        (T_lT_k(v))_\sigma&=dl_{kx}(T_k(v)_{\sigma l})=dl_{kx}dk_x(v_{\sigma lk}) \\
        &=d(lk)_x(v_{\sigma(lk)})=(T_{lk}(v))_\sigma.
    \end{align*}
    Thus these maps $\{T_k\}_{k\in G}$ form a left action.
    By the definition of $H$, 
    $$|T_k(v)|^2_{H,kx}=\sum_{\sigma}|(T_k(v))_{\sigma}|^2_{h_\sigma,kx}=\sum_{\sigma}|dk_x(v_{\sigma k})|^2_{h_\sigma,kx}.$$
    By the definition of $h_\sigma=\sigma^*h$, 
    $$(h_\sigma)_{kx}(dk_xu,dk_xv)=(k^*h_\sigma)_x(u,v)=(h_{\sigma k})_x(u,v).$$
    Thus, 
    $$|dk_x(v_{\sigma k})|^2_{h_\sigma,kx}=|v_{\sigma k}|^2_{h_{\sigma k},x}.$$
    Therefore, 
    $$|T_k(v)|^2_{H,kx}=\sum_{\sigma}|v_{\sigma k}|^2_{h_{\sigma k},x}
    =\sum_{\tau\in G}|v_\tau|^2_{h_\tau,x}=|v|^2_{H,x}.$$
    It is easy to see that $T_{k^{-1}}$ is the inverse of $T_k$.
    Hence, $T_k$ is an isometry of $(E,H)$.
    Moreover, for any $v\in E_x$,
    \begin{align*}
        \Sigma(T_k(v))&=\sum_{\sigma}(T_k(v))_\sigma=\sum_{\sigma}dk_x(v_{\sigma k}) \\
        &=dk_x(\sum_{\tau}v_\tau)=dk_x\Sigma(v),
    \end{align*}
    that is, 
    $\Sigma\circ T_k=dk_x\circ\Sigma$.
    We have the commutative diagram
    \begin{equation*}
        \begin{tikzcd}
            E_x \ar[r,"T_k"] \ar[d,"\Sigma"]& E_{kx} \ar[d,"\Sigma"]      \\
            T_xX \ar[r,"dk_x"]& T_{kx}X.
        \end{tikzcd}
    \end{equation*}
    This means that $\Sigma$ is equivariant.
    For any $w\in T_xX$, by the definition of $\widetilde{h}$, we have 
    $$|dk_x(w)|^2_{\widetilde{h},kx}=\inf\{|W|^2_{H,kx} ~|~\Sigma(W)=dk_x(w) \}.$$
    Note that for any $v\in E_x$ satisfying $\Sigma(v)=w$, we have 
    $$\Sigma(T_k(v))=dk_x(w) \quad \text{and} \quad |T_k(v)|^2_{H,kx}=|v|^2_{H,x}.$$
    This implies that 
    $$|dk_x(w)|^2_{\widetilde{h},kx}\leq\inf\{|v|^2_{H,x} ~|~\Sigma(v)=w \}=|w|^2_{\widetilde{h},x}.$$
    By considering the $k^{-1}$ action, we can obtain the inverse inequality. 
    Therefore, we have 
    $$|dk_x(w)|^2_{\widetilde{h},kx}=|w|^2_{\widetilde{h},x}.$$
    This means that $\widetilde{h}$ is invariant under deck transformations, that is, 
    $\widetilde{h}$ is $G$-invariant.
    
    \medskip
    \noindent\textbf{Step 3: descend $\widetilde{h}$ to a metric on $Y$.}
    
    Note that $p$ is a local biholomorphic map.
    For any $y=p(x)\in Y$ and any $u,v\in T_xX$, define a metric $h_Y$ on $TY$ by
    $$h_Y(dp_x(u),dp_x(v)):=\widetilde{h}(u,v).$$
    For $kx$ with $y=p(kx)$, we have $dp_{kx}=dp_x\circ (dk_x)^{-1}$.
    The invariance of $\widetilde{h}$ shows that the definition is independent of the choice of $x$.
    Thus $h_Y$ is a well-defined metric on $Y$.
    In Step 1, we have shown that $(TX,\widetilde{h})$ is uniformly RC-positive, that is, 
    at every point $x\in X$, there is a nonzero vector $u\in T_xX$ such that for any nonzero vector $v\in T_xX$, one has 
    $$R^{\widetilde{h}}(u,\bar{u},v,\bar{v})>0.$$
    Then for any $y=p(x)\in Y$, take $e=dp_x(u)\in T_yY$. 
    For any nonzero $f\in T_yY$, take $v=dp^{-1}_x(f)\in T_xX$.
    By the definition, we have 
    $$R^{h_Y}(e,\bar{e},f,\bar{f})=R^{\widetilde{h}}(u,\bar{u},v,\bar{v})>0.$$
    This shows that $(TY,h_Y)$ is uniformly RC-positive.
    
    This completes the proof of Lemma~\ref{Lemma:Finite quotient of Yang's metric}.
\end{proof}

We need the following openness of uniform RC-positivity in a proper holomorphic family.
\begin{lemma}\label{Lemma:Openness in a proper holomorphic family}
    Let $f:X\longrightarrow\Delta\subset\CC$ be a proper holomorphic submersion over a disc, with compact fibers $X_t:=f^{-1}(t)$.
    If $T_{X_0}$ admits a uniformly RC-positive Hermitian metric, then so does $T_{X_t}$ for all sufficiently small $|t|$.
\end{lemma}
\begin{proof}
    Consider the relative tangent bundle $V:=T_{X/\Delta}=\ker df$ with $V|_{X_t}=TX_t$.
    Let $h_0$ be a uniformly RC-positive metric on $X_0$. 
    We first extend the given metric $h_0$ on $V|_{X_0}$ to a smooth metric on $V$.
    By the smooth extension theorem \cite[Lemma 10.12]{Lee13}, we can extend $h_0$ to a smooth section $\widehat{h}$ of the real vector bundle of Hermitian forms on $V$, with $\widehat{h}|_{X_0}=h_0.$
    Then $\widehat{h}$ is positive definite on some neighborhood of $X_0$.
    Then by taking a suitable cutoff function and choosing any background Hermitian metric on $V$, we can obtain a smooth Hermitian metric $h$ on $V$ satisfying $h|_{X_0}=h_0$.
    
    For any fixed smooth metric $\omega$ on $X$, consider the function on $X$ defined by
    $$\mu(x):=\sup_{0\neq u\in V_x}\inf_{0\neq v\in V_x}\frac{R^h(u,\bar{u},v,\bar{v})}{|u|^2_\omega|v|^2_h}.$$
    By the uniform RC-positivity of $h_0$, we know that 
    $$\delta:=\inf_{x\in X_0}\mu(x)>0.$$
    Therefore, $$U:=\{\mu(x)>\frac{\delta}{2}\}$$ is an open neighborhood of $X_0$.
    We now show that $X_t\subset U$ for $|t|$ small enough.
    Otherwise, we assume that there exist a sequence $\{t_j\}$ tending to zero and a sequence $\{x_j\}$ with $x_j\in X_{t_j}\setminus U$.
    Let $K$ be a compact neighborhood of zero containing $\{t_j\}$. 
    Then $f^{-1}(K)$ is compact since $f$ is proper.
    Thus there is a subsequence of $\{x_j\}$ converging to a point $x_\infty\in f^{-1}(K)\cap (X\setminus U)$.
    By the continuity of $f$, it is easy to see that $f(x_\infty)=0$. 
    This is a contradiction. 
    Hence, $X_t\subset U$ for $|t|$ small enough.
    
    Finally, since $V|_{X_t}=TX_t$, there exists an $\omega$-unit vector $u\in T_x{X_t}$ such that 
    $$R^{h_t}_x(u,\bar{u},v,\bar{v})\geq\frac{\delta}{2}|v|^2_{h_t} \quad \text{for every } 0\neq v\in T_xX_t \quad \text{and } |t|\ll1,$$
    where $h_t:=h|_{X_t}$.
    
    This completes the proof of Lemma~\ref{Lemma:Openness in a proper holomorphic family}.
    
\end{proof}

We can now prove the uniform RC-positivity of general Hopf surfaces.
Our proof depends heavily on the deformations given by Hasegawa \cite{Has93,Has23}.

\begin{theorem}\label{Theorem:Positivity for Hopf surfaces}
    Let $X$ be a Hopf surface. 
    Then $TX$ admits a uniformly RC-positive Hermitian metric.
\end{theorem}
\begin{proof}
    By Proposition~\ref{Proposition:Hopf class 1},~\ref{Proposition:Hopf class 0}, the conclusion holds if $X$ is a primary Hopf surface. 
    Let $X=W/\Gamma$ be a Hopf surface where $W=\CC^2\setminus\{0\}$ and $\Gamma$ is the covering transformation group of $X$.
    According to \cite[Theorem~2.1, Theorem~2.3, Corollary~2.4]{Has93}, we have the following properties:
    \begin{enumerate}
        \item $\Gamma$ can be expressed as a semi-direct product of an infinite cyclic
        subgroup $\langle F \rangle$ generated by a contraction and a finite normal subgroup $H$, that is, 
        $$\Gamma=H\rtimes\langle F \rangle,$$
        where $H$ is finite and normal and $F$ is a contraction.
        \item There exists an $m\in\NN^*$ such that $F^m$ belongs to the center of $\Gamma$.
    \end{enumerate}
    According to Hasegawa's deformation construction \cite[Theorems~3.3 and 3.7, Lemmas~3.4 and 3.6]{Has93} (restated explicitly in \cite[Section~3]{Has23}), we also have the following property:
    \begin{itemize}
        \item[(3)] $X$ can be successively degenerated into a linear model and a diagonal model through two proper holomorphic families.
    \end{itemize}
    More precisely, there are two complex analytic families such that
    \begin{itemize}
        \item[$\bullet$] the first has central fiber a linear general Hopf surface $X_{\lin}=W/L(\Gamma)$, while every nonzero fiber is biholomorphic to $X$;
        \item[$\bullet$] the second has central fiber a diagonal general Hopf surface $X_\diag=W/\Gamma_\diag$, while every nonzero fiber is biholomorphic to $X_\lin$.
    \end{itemize}
    
    The explicit definition of $L(\Gamma)$ and $\Gamma_\diag$ will not be used in our proof.
    So we omit the definition here and refer to \cite{Has93,Has23} for the definition and some basic properties.
    
    We now show the uniform RC-positivity of $X_\diag$. 
    By \cite{Has23}, 
    $$\Gamma_\diag= H_0\rtimes \langle D \rangle,$$
    where $H_0$ is finite and $D$ is a diagonal contraction.
    Since the conjugation by $D$ induces an automorphism of the finite group $H_0$, it has finite order.
    Therefore, there is $M>0$ such that
    $$D^MhD^{-M}=h \quad \text{for every } h\in H_0.$$
    Then the cyclic group $\langle D^M \rangle$ is a central infinite cyclic subgroup of finite index in $\Gamma_\diag$.
    Define 
    $$\widetilde{X}:=W/\langle D^M \rangle \quad \text{and} \quad G_0=\Gamma_\diag/\langle D^M \rangle.$$
    Then $G_0$ is finite, and we obtain a finite Galois cover
    $$p:\widetilde{X}\longrightarrow X_\diag=\widetilde{X}/G_0.$$
    It is \'etale since the induced finite-group action is free. 
    Actually, if the class of $\gamma\in\Gamma_\diag$ fixes the class $[z]\in\widetilde{X}$ for $z\in W$, then $\gamma z=D^{Ml}z$ for some $l$.
    Hence $D^{-Ml}\gamma$ fixes $z\in W$. 
    Thus it is the identity deck transformation.
    
    Replacing the contraction $D^M$ by its inverse and interchanging the two coordinates if necessary, we can write
    $$(D^M)^{-1}(z,w)=(az,bw) \quad \text{where } |a|\geq|b|>1.$$
    Therefore, $\widetilde{X}$ is a primary Hopf surface $H_{a,b}$ of class 1.
    For this special Hopf surface, by the computation of Proposition~\ref{Proposition:Hopf class 1} (\cite{Yan17}), we can apply Lemma~\ref{Lemma:Finite quotient of Yang's metric} to $\widetilde{X}$.
    Then we prove that the tangent bundle of $X_\diag$ is uniformly RC-positive.
    
    By applying Lemma~\ref{Lemma:Openness in a proper holomorphic family} first to Hasegawa's second family, 
    we know that every sufficiently small nonzero fiber is uniformly RC-positive.
    Since all nonzero fibers are biholomorphic to $X_\lin$, we obtain that $TX_\lin$ is uniformly
    RC-positive. 
    Similarly, by applying Lemma~\ref{Lemma:Openness in a proper holomorphic family} to Hasegawa's first family, we obtain that $TX$ is uniformly RC-positive.
    
    This completes the proof of Theorem~\ref{Theorem:Positivity for Hopf surfaces}.
    
\end{proof}

\subsection{A classification theorem for compact complex surfaces}

It is well-known that every compact complex surface can be obtained from a minimal surface by blow-ups.
We now prove the following classification theorem for compact complex surfaces. 

\begin{theorem}[{=Theorem~\ref{Main Thm 2}}] \label{Main Thm 2'}
    Let $X$ be a compact connected complex surface.
    \begin{enumerate}
        \item If $TX$ is uniformly RC-positive, then $\kappa(X)=-\infty$, and the minimal model of $X$ is either rational or of class VII.  
        An irrational ruled surface cannot occur.
        \item If $X$ is K\"ahler, then
        $TX$ is uniformly RC-positive if and only if $X$ is rational.
        \item Every Hopf surface has uniformly RC-positive tangent bundle.  
        The same holds after any finite sequence of point blow-ups.
        \item No Inoue-Bombieri surface, and no finite sequence of point blow-ups of one, has uniformly RC-positive tangent bundle.
        \item Among minimal class VII surfaces with $b_2=0$, the uniformly RC-positive ones are exactly the Hopf surfaces. 
    \end{enumerate}
\end{theorem}
\begin{proof}
    (1) If $TX$ is uniformly RC-positive, we have 
    $$H^0(X,\Omega^p_X)=0 \quad \text{for } p=1,2,$$
    and $K_X$ is not pseudo-effective by Proposition~\ref{Proposition:Basic property for uniformly RC-positive vector bundle}, \ref{Proposition:Yang's vanishing}.
    In particular, $\kappa(X)=-\infty$.
    By the Enriques--Kodaira classification \cite{BPV84}, a minimal compact complex surface with $\kappa=-\infty$ is rational, is a ruled surface of genus at least 1, or is of class VII.
    
    However, a ruled surface over a curve $C$ of genus $g(C)\geq 1$ carries the pull-backs of nonzero forms in $H^0(C,K_C)$. 
    Since holomorphic one-forms are unchanged by point blow-ups of smooth surfaces, every surface bimeromorphic to such a ruled surface carries a holomorphic one-form.
    Thus the minimal model of $X$ is either rational or of class VII.
    
    (2) If $X$ is K\"ahler, $X$ is rationally connected if and only if $TX$ is uniformly RC-positive by Theorem~\ref{Main Thm 0}.
    
    (3) This is an immediate consequence of Theorems~\ref{Theorem:Positivity for Hopf surfaces} and \ref{Main Thm 1}.
    
    (4) For every Inoue-Bombieri surface $X$, the canonical bundle $K_X$ has a smooth Hermitian metric with semi-positive curvature (see \cite[Lemma~3]{AD08} and \cite[Remark~4.2]{Tel06}). 
    In particular, $K_X$ is pseudo-effective.
    
    If $\pi:\widetilde{X}\longrightarrow X$ is a point blow-up and $E$ is the exceptional divisor, then by the formula
    $$K_{\widetilde{X}}=\pi^*K_X\otimes\mathcal{O}(E),$$
    we know that $K_{\widetilde{X}}$ is pseudo-effective (see also \cite[Remark 4.2]{Tel06}).
    Therefore, by Proposition~\ref{Proposition:Basic property for uniformly RC-positive vector bundle}, $TX$ and $T\widetilde{X}$ are not uniformly RC-positive.
    
    (5) According to the work of Bogomolov, Li, Yau, Zheng and Teleman (\cite{Bog76,Bog82,LY87,LYZ90,Tel94}), minimal surfaces of class VII with $b_2=0$ are precisely Hopf surfaces and Inoue--Bombieri surfaces.
    Thus the conclusion follows from (3) and (4).
    
\end{proof}

\section{Surfaces of class $\VII_0^+$} \label{Sec:Classification 2}

In this section, we consider minimal surfaces of class $\VII_0^+$.

Here is a simple observation on the positivity of anticanonical bundles.
Teleman \cite[Proposition 4.5]{Tel06} proved that on every minimal class $\VII$ surface $X$ with $b_2(X)>0$, the Bott--Chern class $c_1^\BC(K_X)$ is not pseudo-effective.
Then by Yang's result (Theorem~\ref{Theorem:Characterization of RC-positive line bundle}), $-K_X$ is (uniformly) RC-positive.

A Kato surface $X$, introduced by Kato \cite{Kat77}, is a minimal compact complex surface with positive second Betti number $b_2(X)>0$ containing a global spherical shell (GSS),
that is, there exists an open set $U\subset X$ such that $X\setminus U$ is connected and $U$ is biholomorphic to an open neighborhood of the sphere $S^3$ in $\CC^2\setminus\{0\}$.
The GSS conjecture predicts that a surface of class $\VII_0^+$ is a Kato surface. 
For a recent result on the GSS conjecture under an additional foliation hypothesis, see \cite{KS26}, which proves that minimal class $\VII$ surfaces with $b_2(X)=3$ and two distinct holomorphic foliations are Kato surfaces.

Hence we mainly consider the positivity of Kato surfaces in this section.

\subsection{Uniform RC-positivity}
In this subsection, we prove that the tangent bundle of any Kato surface is uniformly RC-positive. 

Dloussky and Oeljeklaus \cite[Main Theorem in page 1504, Theorem~3.1 in page 1529]{DO99} showed that 
for any Kato surface $X$, there exists a singular holomorphic foliation $\mathcal{F}$ with a tangent line bundle $T\mathcal{F}$ (rather than merely a rank-one torsion-free sheaf) and an exact sequence
\begin{equation}
    0\rightarrow T\mathcal{F}\rightarrow TX\rightarrow N\mathcal{F}\otimes \mathcal{I}_Z\rightarrow 0,
\end{equation}
where $N\mathcal{F}:=(TX/T\mathcal{F})^{**}$ is the normal bundle and $Z$ is a zero-dimensional complex space.
Motivated by this structure, we first prove the following lemma.
\begin{lemma}\label{Lemma:Uniformly RC-positive for Kato surfaces}
    Let $X$ be a compact complex surface and let $E$ be a holomorphic vector bundle of rank two.
    Assume that there is an exact sequence 
    \begin{equation}
        0\rightarrow L_1\rightarrow E\rightarrow L_2\otimes \mathcal{I}_Z\rightarrow 0,
    \end{equation}
    where $L_1,L_2$ are line bundles and $Z$ is a zero-dimensional complex space.
    If there exists a Gauduchon metric $\omega$ on $X$ such that 
    $$\deg_\omega L_1>0 \quad \text{and} \quad \deg_\omega L_2>0,$$
    then $E$ is uniformly RC-positive.
\end{lemma}
\begin{proof}
    
    \medskip
    \noindent\textbf{Step 1: construct a metric on $L_1\oplus L_2$.}
    
    Since each $L_i, i=1,2$ is a line bundle with $\deg_\omega L_i>0$, we have $\mu_L(L_i,\omega)>0$ (see the definition before Theorem~\ref{Theorem:Mean curvature positivity and HN-positivity}).
    Then by Theorem~\ref{Theorem:Mean curvature positivity and HN-positivity}, there is a Hermitian metric $h_i$ on $L_i$ such that
    $\tr_\omega R^{h_i}>0$ (see also \cite[proof of Theorem~1.5]{Yan19}). 
    
    Let $\lambda_a$ and $\lambda_b$ be two eigenvalues of $R^{h_1}$ at some point $x\in X$. 
    Then we have $\lambda_a+\lambda_b>0$.
    If $\lambda_a>0$ and $\lambda_b\geq0$, then for any tangent vector $u$ in the eigenspace associated with $\lambda_b$, we have $R^{h_1}_x(u,\bar{u})\geq0$. 
    The equality holds only if $\lambda_b=0$ and $u$ is in the eigenspace associated with $\lambda_b$.
    If $\lambda_a>0$ and $\lambda_b<0$, we take an orthogonal basis $e_a$ and $e_b$ associated with $\lambda_a$ and $\lambda_b$.
    Take $u=u_ae_a+u_be_b$ with $|u_a|^2+|u_b|^2=1$.
    Then we have 
    $$R^{h_1}_x(u,\bar{u})>0 \Longleftrightarrow |u_a|^2>\frac{-\lambda_b}{\lambda_a-\lambda_b},
    \quad \text{where} \quad \frac{-\lambda_b}{\lambda_a-\lambda_b}\in (0,\frac{1}{2}).$$
    The same conclusion holds for $R^{h_2}$.
    Since 
    $$2\Vol\{|u_a|^2+|u_b|^2=1~|~|u_a|^2>c \text{ for some constant } c<\frac{1}{2}\}>
    \Vol\{|u_a|^2+|u_b|^2=1\},$$
    there is a nonzero tangent vector $u_x$ such that 
    $$R^{h_1}_x(u_x,\bar{u}_x)>0 \quad \text{and} \quad R^{h_2}_x(u_x,\bar{u}_x)>0.$$
    Then by continuity and compactness, 
    $$\delta:=\inf_X\sup_{|u|_\omega=1}\inf\{R^{h_1}_x(u,\bar{u}),R^{h_2}_x(u,\bar{u})\}>0.$$
    This implies that the vector bundle 
    $$(L_1\oplus L_2,h_1\oplus h_2)$$
    is uniformly RC-positive.
    
    \medskip
    \noindent\textbf{Step 2: construct a metric on $E$ over $X\setminus Z$.}
    
    By the exact sequence, we have a smooth (not holomorphic) isomorphism on $U:=X\setminus Z$:
    $$E|_U\simeq L_1|_U\oplus L_2|_U.$$
    Let $\bar\partial_E,\bar\partial_{L_1},\bar\partial_{L_2}$ be complex structures on $E,L_1,L_2$ respectively.
    Then in terms of the second fundamental form $\beta$, we can write 
    \begin{equation*}
        \bar\partial_E=
        \begin{pmatrix}
            \bar\partial_{L_1} & -\beta^* \\
            0 & \bar\partial_{L_2}
        \end{pmatrix}.
    \end{equation*}
    Put $G_\varepsilon:=\diag(\varepsilon,1)$ as a smooth linear automorphism on $L_1\oplus L_2$ for every $\varepsilon>0$ and define a new complex structure on $L_1\oplus L_2$ by 
    \begin{equation*}
        \bar\partial_\varepsilon:=G_\varepsilon\bar\partial_E G_\varepsilon^{-1}=
        \begin{pmatrix}
            \bar\partial_{L_1} & -\varepsilon \beta^* \\
            0 & \bar\partial_{L_2}
        \end{pmatrix}.
    \end{equation*}
    Then $\bar\partial_\varepsilon\circ G_\varepsilon=G_\varepsilon\circ\bar\partial_E$.
    This means that $G_\varepsilon$ is a holomorphic bundle isomorphism from $(E,\bar\partial_E)$ to $(L_1\oplus L_2,\bar\partial_\varepsilon)$.
    
    Equip $(L_1\oplus L_2,\bar\partial_\varepsilon)$ with the fixed split metric $h_0:=h_1\oplus h_2$.
    Denote by $D_\varepsilon$ and $R_\varepsilon$ the Chern connection and Chern curvature associated with $(L_1\oplus L_2,\bar\partial_\varepsilon,h_0)$.
    Denote $D_0$ and $R_0$ the Chern connection and Chern curvature associated with $(L_1\oplus L_2,\bar\partial_0,h_0)$, where $\bar\partial_0:=\bar\partial_{L_1}\oplus\bar\partial_{L_2}$.
    Then on every compact subset $K\Subset U$, we have
    $$D_\varepsilon=D_0+A_\varepsilon \quad \text{with} \quad
    A_\varepsilon=O(\varepsilon)\rightarrow 0 \text{ in } C^\infty(K),$$
    and 
    \begin{equation}
        R_\varepsilon=R_0+d^{D_0}A_\varepsilon+A_\varepsilon\wedge A_\varepsilon\rightarrow R_0 \text{ in } C^\infty(K).    
    \end{equation}
    Define $h_\varepsilon=G_\varepsilon^*h_0$ on $E|_U$.
    Then for any $x\in U, u\in T_xX$ and $v\in E_x$, we have 
    $$R^{h_\varepsilon}_x(u,\bar{u},v,\bar{v})=R_{\varepsilon,x}(u,\bar{u},G_\varepsilon v,\bar{G_\varepsilon v}).$$
    Then on every compact subset $K\Subset U$, there is $\varepsilon>0$ such that 
    $$\inf_K\sup_{|u|_\omega=1}\inf_{|v|_{h_\varepsilon}=1}R^{h_\varepsilon}_x(u,\bar{u},v,\bar{v})\geq\frac{\delta}{2}.$$
    
    \medskip
    \noindent\textbf{Step 3: construct a metric on $E$ over $X$.}
    
    We adapt the conformal trick used in the proof of Theorem~\ref{Main Thm 1}.
    
    Choose sufficiently small disjoint holomorphic coordinate balls $(B_j;(z_j^1,z_j^2))$ around the finitely many points of $Z$.
    Set $\rho_j=|z_j^1|^2+|z_j^2|^2$.
    After shrinking the balls, we may assume 
    $c_0\omega\leq\ddbar\rho_j\leq C_0\omega$ on $\bigcup_j\overline{B_j}$ for some positive constants $C_0>c_0>0$.
    Choose a constant $B_0>0$ such that, at every point $x$, for every $\omega$-unit base vector $u$ and every $v\in (L_1\oplus L_2)_x$, 
    $$R^{h_0}_x(u,\bar{u},v,\bar{v})\geq-B_0|v|^2_{h_0}.$$
    
    By applying Lemma~\ref{Lemma:auxiliary function 2}, for the given constants 
    $$M>\frac{B_0+1}{c_0} \quad \text{and} \quad 0<\eta<\frac{\delta}{4C_0},$$
    there exist constants $S>2$ and $L>0$ satisfying the requirement in Lemma~\ref{Lemma:auxiliary function 2}.
    Choose a small constant $a>0$ such that $\{\rho_j<Sa\}\Subset B_j$.

    Consider the compact set $K:=X\setminus\bigcup_j\{\rho_j<a\}\Subset U$.
    By Step 2, there is $\varepsilon>0$ such that 
    \begin{itemize}
        \item[$\bullet$] at each point $x\in K$, there is an $\omega$-unit base vector $u$ such that
        $$R^{h_\varepsilon}(u,\bar{u},v,\bar{v})\geq\frac{\delta}{2}|v|^2_{h_\varepsilon} \quad \text{for every } v\in E_x.$$
        \item[$\bullet$] at each point $x\in K$, for every $\omega$-unit base vector $u$,
        $$R^{h_\varepsilon}(u,\bar{u},v,\bar{v})\geq-(B_0+1)|v|^2_{h_\varepsilon} \quad \text{for every } v\in E_x.$$
    \end{itemize}
    Let $k$ be a global smooth metric on $E$ over $X$. 
    Take a cut-off function 
    $\chi\in C^\infty([0,\infty),[0,1])$ satisfying
    $$\chi(s)=1\quad(s\le a/3) \quad \text{and} \quad \chi(s)=0\quad(s\ge2a/3).$$
    Then the function $\chi(\rho_j)$ on $B_j$ vanishes on a neighborhood of its boundary.
    Extend it by zero to a global smooth function
    $\widetilde\chi$ on $X$, and set
    $$\widetilde{h}:=\widetilde\chi k+(1-\widetilde\chi)h_\varepsilon.$$
    Then $\widetilde{h}$ is a smooth metric on $E$ over $X$ and 
    $$\widetilde{h}=h_\varepsilon \quad \text{on} \quad X\setminus\bigcup_j\{\rho_j<2a/3\}.$$
    
    Consider the compact set $\bigcup_j\{\rho_j\leq a\}.$
    There is a constant $B_a>0$ such that at each point $x\in\bigcup_j\{\rho_j\leq a\}$, for every $\omega$-unit base vector $u$,
    $$R^{\widetilde{h}}(u,\bar{u},v,\bar{v})\geq-(B_a+1)|v|^2_{\widetilde{h}} \quad \text{for every } v\in E_x.$$
    Now choose a constant 
    $$A>\max\{M+1,\frac{B_a+1}{c_0}\}.$$
    Then for constants $A$ and $a$, there is a smooth function $\lambda$ satisfying Lemma~\ref{Lemma:auxiliary function 2}(1)--(5).
    Define a smooth function
    $$f(\tau):=\int_{0}^{\tau}\frac{\lambda(s)}{s}ds.$$
    Note that $\lambda=0$ on a neighborhood of $[Sa,\infty)$. 
    Thus $f$ is constant there.
    This implies that 
    $f(\rho_j)=f(Sa)$ is constant on a neighborhood of the boundary of $B_j$ since $\{\rho_j<Sa\}\Subset B_j$.
    Then we define
    \begin{equation*}
        F:=
        \begin{cases}
            f(\rho_j),&\text{on } B_j,\\
            f(Sa),&\text{on }X\setminus\bigcup_jB_j.
        \end{cases}
    \end{equation*}
    The argument above shows that $F$ is a globally defined smooth function on $X$.
    Define a new Hermitian metric on $E$ by 
    $$H=e^{-F}\widetilde{h}.$$
    Then by the same argument as in Step 3 in the proof of Theorem~\ref{Main Thm 1}, one can see that $H$ is uniformly RC-positive.
    
    This completes the proof of Lemma~\ref{Lemma:Uniformly RC-positive for Kato surfaces}.
    
\end{proof}

Now we can prove the positivity of Kato surfaces.
\begin{theorem}\label{Main Thm 3(2)}
    Let $X$ be a Kato surface. 
    Then $TX$ is uniformly RC-positive.
\end{theorem}
\begin{proof}
    Consider the singular holomorphic foliation $\mathcal{F}$ constructed by Dloussky--Oeljeklaus in \cite{DO99}.
    Let $T\mathcal{F}$ be the tangent line bundle of $\mathcal{F}$. 
    Then we have an exact sequence
    \begin{equation}\label{eq:Exact sequence of TF}
        0\rightarrow T\mathcal{F}\rightarrow TX\rightarrow N\mathcal{F}\otimes \mathcal{I}_Z\rightarrow 0,
    \end{equation}
    where $N\mathcal{F}$ is the normal bundle and $Z$ is a zero-dimensional complex space.
    (We briefly explain the argument in Lemma~\ref{Lemma:The endpoint foliations}.)
    By Lemma~\ref{Lemma:Uniformly RC-positive for Kato surfaces}, it is sufficient to find a Gauduchon metric such that 
    $$\deg_\omega T\mathcal{F}>0 \quad \text{and} \quad \deg_\omega N\mathcal{F}>0.$$
    
    \medskip
    \noindent\textbf{Step 1: recall some known results and constructions.}
    
    A surface $X$ with a GSS contains exactly $n:=b_2(X)>0$ rational curves $D_1,\ldots,D_n$, each of which is smooth or has a double point. 
    Dloussky and Oeljeklaus \cite[Section 0, page 1504]{DO99} defined 
    $$\sigma:=-\sum_{i=1}^{n}D_i^2+2\#\{\text{double points}\} \quad \text{satisfying} \quad 2n\leq\sigma\leq 3n.$$
    Write
    $$a=c_1(T\mathcal{F}), \quad b=c_1(N\mathcal{F}), \quad \text{and} \quad k=c_1(K_X)=-(a+b).$$
    Then by Baum--Bott formulas and Camacho--Sad formula, we have (see \cite[Section 3, formulas (4)--(6)]{DO99})
    \begin{equation}
        a^2=\sigma-3n\leq0, \quad b^2=2n-\sigma\leq0, \quad \text{and} \quad a\cdot b=0.
    \end{equation}
    This implies that 
    \begin{equation}\label{eq:computation of ka and kb}
        k\cdot a=3n-\sigma\geq0, \quad \text{and} \quad k\cdot b=\sigma-2n\geq0.
    \end{equation}
    
    We now recall Teleman's result on the Gauduchon cone which will be used later (see \cite[Remark~2.4, Proposition~2.5, Theorem~2.6]{Tel06}).
    Note that $b_1(X)=1$ is odd.
    The space of Gauduchon metrics on $X$ is connected, and the orientation of the ``exact line''
    $$\Gamma(X)\subset H^{1,1}_{\BC}(X,\RR)$$
    induced by $\langle \cdot,\omega_g\rangle$, where $\omega_g$ is any Gauduchon metric, is well defined. 
    Let $\gamma_0$ be a positive generator of this line. One has
    $$\langle \gamma_0,[\omega_g]\rangle>0$$
    for every Gauduchon metric $\omega_g$ on $X$.
    The Bott-Chern class $\gamma_0$ is represented by an exact positive current.
    In particular, the de Rham class of $\gamma_0$ is zero, and by Stokes' formula, 
    $$\gamma_0^2=0 \quad \text{and} \quad \gamma_0\cdot c_1^{\BC}(L)=0\quad \text{ for any holomorphic line bundle } L.$$
    Then by Buchdahl's ampleness criterion,
    a class $h\in \mathcal{H}(X)$ is represented by a Gauduchon metric if and only if
    \begin{equation}\label{eq:Gauduchon criterion}
        h^2>0, \quad h\cdot\gamma_0>0, \quad \text{and} \quad h\cdot C>0
    \end{equation}
    for every irreducible effective curve $C$ with $C^2<0$.
    Here $\mathcal{H}(X)$ is given by 
    $$\mathcal{H}(X):=\frac{\ker \{\ddbar: A^{1,1}(X,\RR)\rightarrow A^{2,2}(X,\RR)\}}{\operatorname{Im} \{\ddbar: A^{0}(X,\RR)\rightarrow A^{1,1}(X,\RR)\}}
    \supset H^{1,1}_{\BC}(X,\RR).$$
    
    \medskip
    \noindent\textbf{Step 2: compute the degree of $T\mathcal{F}$ and $N\mathcal{F}$.}
    
    Let $\omega_0$ be a Gauduchon metric on $X$.
    For every irreducible curve $C$ with $C^2<0$, adjunction formula implies (see \cite[Chapter II, Section 11]{BPV84})
    $$K_X\cdot C=2p_a(C)-2-C^2,$$
    where $p_a(C)=1-\chi(\mathcal{O}_C)$ is the  arithmetic genus of $C$.
    If $p_a(C)\geq 1$, we have $K_X\cdot C\geq 0$.
    If $p_a(C)=0$, then $C$ is a smooth rational curve. 
    Since $X$ is minimal, we have $C^2\leq -2$.
    Therefore, $K_X\cdot C\geq 0$.
    Hence for any $s,t\geq0$, the class
    $$h_{t,s}:=[\omega_0]+tc_1^{\BC}(K_X)+s\gamma_0$$
    satisfies 
    $$h_{t,s}\cdot C=\int_C\omega_0+tK_X\cdot C>0 \quad \text{and} \quad h_{t,s}\cdot \gamma_0=[\omega_0]\cdot\gamma_0>0$$
    since $\gamma_0$ is represented by an exact positive current.
    Moreover, for any fixed $t$, one can choose $s$ large enough such that
    $$h^2_{t,s}=([\omega_0]+tc_1^{\BC}(K_X))^2+2s[\omega_0]\cdot\gamma_0>0.$$
    Then by formula~(\ref{eq:Gauduchon criterion}), $h_{t,s}$ is represented by a Gauduchon metric $\omega_{t,s}$.
    Note that for any line bundle $L$, we have $\gamma_0\cdot c_1^{\BC}(L)=0$. 
    Therefore,
    \begin{equation}
        \begin{aligned}
            \deg_{\omega_{t,s}}L&=\int_{X}c_1^{\BC}(L)\wedge\omega_{t,s}=\int_{X}c_1^{\BC}(L)\wedge(\omega_0+tc_1^{\BC}(K_X)+s\gamma_0) \\
            &=\int_{X}c_1^{\BC}(L)\wedge\omega_0+t\int_{X}c_1^{\BC}(L)\wedge c_1^{\BC}(K_X)\\
            &=\deg_{\omega_0}L+tc_1(L)\cdot c_1(K_X)=\deg_{\omega_0}L+tk\cdot c_1(L)
        \end{aligned}
    \end{equation}
    is independent of $s$.
    
    Now we consider three different cases.
    
    \medskip
    \noindent\textbf{Case I: $\sigma=2n$.}
    
    In this case, the surface is called an Enoki surface \cite{DK98,Eno81}, and we have $b^2=0$.
    Since $b_1(X)=1$ is odd and $\kappa(X)=-\infty$, we know that the intersection form on $H^2(X,\RR)$ is negative definite by \cite[Theorem 2.6(iii), Theorem 2.9, and Theorem 2.13 in Chapter IV]{BPV84}.
    Hence $b=0$ in de Rham cohomology.
    We also have $k\cdot a=n>0$ in this case.
    Hence, we can choose $t\gg 1$ and $s\gg 1$ such that $h_{t,s}$ is represented by a Gauduchon metric $\omega_{t,s}$ and $\deg_{\omega_{t,s}}T\mathcal{F}>0$.
    
    According to \cite[Lemma~6.6]{AD23} and \cite[Main Theorem and Proposition 3.4, 3.5]{Eno81}, 
    for any Enoki surface $X$, there is a cycle of rational curves $D$ in $X$ with $D^2=0$ and 
    $$\mathcal{O}_X(D)\simeq L_\mu, \quad \mu\in\CC^*, \quad 0<|\mu|<1,$$
    where $L_\mu\in \Pic^0(X)\simeq\CC^*$ is a flat line bundle.
    Therefore, $c_1(\mathcal{O}_X(D))=0$.
    Moreover, $X$ admits one singular holomorphic foliation $\mathcal{F}$ defined by a logarithmic $1$-form $\theta\in H^0(X,\Omega^1_X(\log D))$ (see also \cite{DK98} and \cite[Section 2.1, page 1512]{DO99}).
    The form is of positive type (see \cite[Definition~6.2]{AD23}) by \cite[Remark~6.9]{AD23}.
    In particular, its residue on every component of $D$ is nonzero.
    Its zero set is either empty or zero-dimensional by \cite[Lemma~6.15 and the paragraph following it]{AD23}.
    
    On the other hand, Dloussky--Oeljeklaus proved that in this case the relevant singular holomorphic foliation is unique (see \cite[Section~2.1, page 1512 and Theorem 5.5(i)]{DO99}). 
    Therefore, by Lemma~\ref{Lemma:The endpoint foliations} below, the foliation defined by $\theta$ is the foliation $\mathcal{F}$ in formula~(\ref{eq:Exact sequence of TF}), 
    and we have $N\mathcal{F}\simeq\mathcal{O}_X(D)$.
    Thus for the Gauduchon metric $\omega_{t,s}$ chosen before,
    $$\deg_{\omega_{t,s}} N\mathcal{F}=\int_{D}\omega_{t,s}>0.$$
    
    \medskip
    \noindent\textbf{Case II: $2n<\sigma<3n$.}
    
    In this case, we have $k\cdot a>0$ and $k\cdot b>0$ by formula~(\ref{eq:computation of ka and kb}). 
    Therefore, we can choose $t\gg 1$ and $s\gg 1$ such that $h_{t,s}$ is represented by a Gauduchon metric $\omega_{t,s}$ 
    and $\deg_{\omega_{t,s}}L>0$ for both $L=T\mathcal{F}$ and $L=N\mathcal{F}$.
    Then we find a Gauduchon metric $\omega=\omega_{t,s}$ satisfying
    $$\deg_\omega T\mathcal{F}>0 \quad \text{and} \quad \deg_\omega N\mathcal{F}>0.$$
    
    \medskip
    \noindent\textbf{Case III: $\sigma=3n$.}
    
    In this case, the surface is called an Inoue--Hirzebruch surface (\cite{Dlo88-a,Dlo88-b,Ino77,Nak84}), and we have $a^2=0$.
    Similarly to Case I, $a=0$ in de Rham cohomology, and $k\cdot b=n>0$.
    We first specify the foliation and then choose the Gauduchon metric..
    
    According to \cite{Dlo88-a,Dlo88-b,Ino77} (see also \cite[Section~2.2, pp.\,1525--1528]{DO99}), 
    any Inoue--Hirzebruch surface can be obtained from the contracting monomial germ in the GSS construction
    $$F(z_1,z_2)=(z_1^pz_2^q,z_1^rz_2^s) \quad \text{with } p,q,r,s\geq0 \quad \text{and}\quad 
    A= \begin{pmatrix}
        p & q \\ r & s
    \end{pmatrix} \in GL(2,\mathbb{Z}).$$
    The matrix $A$ has two eigenvalues
    $$\lambda_1>1 \quad\text{and}\quad 0<|\lambda_2|<1, \quad \text{with } \lambda_1\lambda_2=\det A=\pm 1.$$ 
    
    Dloussky and Oeljeklaus also defined (\cite[Section 2.2, formula ($\bigstar$)]{DO99}) twisted meromorphic $1$-forms with logarithmic
    poles 
    $$\theta_i(z)=\frac{a_i}{z_1}dz_1+\frac{b_i}{z_2}dz_2\in H^0(X,\Omega^1_X(\log D)\otimes L_{\lambda_i})$$
    where
    $(a_i,b_i)^t$ is an eigenvector of $A^t$ corresponding to $\lambda_i$, $i=1,2$.
    Here $D$ is the reduced maximal divisor, namely the union of all compact rational curves. 
    Actually, $D$ is the sum of one or two cycles (\cite[Example~4.4]{DO99}).
    Let $\pi:\widehat{X}\rightarrow X$ be the natural infinite cyclic covering in the GSS construction.
    Fix its deck generator $\gamma$ with
    the convention compatible with the contracting germ F in \cite[Lemma~4.3]{AD23} (cf. \cite[Definition~1.1 and Lemma~1.3]{DO99}).
    For $\mu\in\mathbb{C}^*$, set
    $$L_\mu:=(\widehat{X}\times\mathbb{C}/((x,v)\sim(\gamma x,\mu v))).$$
    The positive real characters satisfy (see \cite[Remarks 3.12-3.13, Example 4.5, and Lemma 6.7]{AD23})
    $$\mu>1 \Longleftrightarrow \deg_\omega L_\mu<0 \text{ for every Gauduchon metric }\omega.$$
    Note that
    $$\theta_1\wedge\theta_2=\det
    \begin{pmatrix}
        a_1 & b_1 \\ a_2 & b_2
    \end{pmatrix}
    \frac{dz_1}{z_1}\wedge\frac{dz_2}{z_2}\neq 0.$$
    Thus $\theta_1\wedge\theta_2$ is a nonzero holomorphic section of $K_X\otimes\mathcal{O}_X(D)\otimes L_{\det A}$.
    Let $E$ be its effective zero divisor.
    By \cite[Section 4, opening paragraph in page 1532]{DO99}, the only irreducible curves on $X$ are the $n=b_2(X)$ components $D_1,\ldots,D_n$ of $D$, and their classes generate $H_2(X,\mathbb{Q})$.
    Therefore, $E=\sum_{j=1}^nm_jD_j$ for some $m_j\geq 0$.
    
    Dloussky--Oeljeklaus showed that $D$ is numerically anticanonical \cite[Definition~4.1, Example~4.4]{DO99}, which means that there is a flat line bundle $M$ such that
    $$K_X^{-1}\otimes M\simeq\mathcal{O}_X(D).$$
    Since $L_{\det A}$ is also flat, we obtain
    $$c_1(K_X\otimes\mathcal{O}_X(D)\otimes L_{\det A})=0.$$
    Therefore, $c_1(\mathcal{O}_X(E))=0$.
    This implies that $E=0$.
    Thus $\theta_1\wedge\theta_2$ is nowhere zero. 
    Hence we obtain
    $$K_X\otimes\mathcal{O}_X(D)\otimes L_{\det A}\simeq\mathcal{O}_X.$$
    
    Let $\mathcal{F}$ be the foliation defined by $\theta_1$.
    By \cite[Definition~6.2, Lemma~6.7 and Remark~6.9]{AD23}, the form $\theta_1$ is of positive type.
    Thus its residue is nonzero on every component of $D$, and 
    its zero set away from $D$ is empty or
    zero-dimensional by \cite[Lemma~6.15]{AD23}.
    By applying Lemma~\ref{Lemma:The endpoint foliations}, we have 
    \begin{equation*}
        0\rightarrow T\mathcal{F}\rightarrow TX\rightarrow N\mathcal{F}\otimes \mathcal{I}_Z\rightarrow 0,
    \end{equation*}
    and 
    $$N\mathcal{F}\simeq\mathcal{O}_X(D)\otimes L_{\lambda_1}.$$
    On the other hand, 
    $$-K_X\simeq T\mathcal{F}\otimes N\mathcal{F}.$$
    Combining the argument above, we obtain
    \begin{align*}
        T\mathcal{F}&\simeq\mathcal{O}_X(D)\otimes L_{\det A}\otimes(\mathcal{O}_X(D)\otimes L_{\lambda_1})^*\\
        &=L_{\det A/\lambda_1} \\
        &=L_{\lambda_2}.
    \end{align*}
    By the character convention fixed above and the inequality $\lambda_1>1$, we have
    $\deg_\omega L_{\lambda_1}<0$ for every Gauduchon metric $\omega$.
    Note that $L_{\lambda_2}=L_{\det A}\otimes L_{\lambda_1}^*$.
    If $\det A=1$, then $L_{\det A}$ is trivial. 
    If $\det A=-1$, then $L_{-1}$ admits a flat Hermitian metric since $-1\in U(1)$.
    Thus $$\deg_\omega L_{\det A}=0.$$
    This shows that 
    $$\deg_\omega T\mathcal{F}=-\deg_\omega L_{\lambda_1}>0.$$
    Then for this foliation, we can find a Gauduchon metric $\omega=\omega_{t,s}$ satisfying
    $$\deg_\omega T\mathcal{F}>0 \quad \text{and} \quad \deg_\omega N\mathcal{F}>0.$$
    
    By applying Lemma~\ref{Lemma:Uniformly RC-positive for Kato surfaces}, $TX$ is uniformly RC-positive.
    
    This completes the proof of Theorem~\ref{Main Thm 3(2)}.
    
\end{proof}

\begin{lemma}\label{Lemma:The endpoint foliations}
    Let $D$ be a reduced normal crossing divisor on a compact complex surface $X$ and $L$ a line bundle over $X$.
    Let $\theta\in H^0(X,\Omega^1_X(\log D)\otimes L)$ be nonzero.
    Suppose that the zeros of $\theta$ away from $D$ are isolated and that the residue of $\theta$ along every irreducible component of $D$ is nonzero. 
    Consider the holomorphic morphism $\widetilde{\theta}\in H^0(X,\Omega^1_X\otimes\mathcal{O}_X(D)\otimes L)=\Hom(TX,\mathcal{O}_X(D)\otimes L)$ defined by
    $$\widetilde{\theta}:TX\longrightarrow \mathcal{O}_X(D)\otimes L, \quad u\longmapsto \theta(u),$$
    induced by $\theta$ via the natural inclusion $\Omega^1_X(\log D)\hookrightarrow\Omega^1_X\otimes\mathcal{O}_X(D)$.
    Then there is a singular holomorphic foliation $\mathcal{F}$ with a tangent line bundle and an exact sequence
    \begin{equation}
        0\rightarrow T\mathcal{F}\rightarrow TX\rightarrow N\mathcal{F}\otimes \mathcal{I}_Z\rightarrow 0,
    \end{equation}
    where $Z$ is a zero-dimensional complex space and $N\mathcal{F}:=(TX/T\mathcal{F})^{**}$ is the normal bundle satisfying
    $$N\mathcal{F}\simeq\mathcal{O}_X(D)\otimes L.$$
\end{lemma}
\begin{proof}
    Near a smooth point of $D$, we may assume that $D=\{z_1=0\}$. 
    Let $e_L$ be a local frame of $L$.
    Write $$\theta=(a\frac{dz_1}{z_1}+b\,dz_2)\otimes e_L=(adz_1+bz_1dz_2)\otimes\frac{1}{z_1}\otimes e_L.$$
    Thus the coefficient row of $\widetilde{\theta}$ is $(a,bz_1)$, and its kernel is generated by
    $$-z_1b\frac{\partial}{\partial z_1}+a\frac{\partial}{\partial z_2}.$$
    
    Near a normal-crossing point of $D$, we may assume that $D=\{z_1z_2=0\}$. 
    Write 
    $$\theta=(a\frac{dz_1}{z_1}+b\frac{dz_2}{z_2})\otimes e_L=(az_2dz_1+bz_1dz_2)\otimes\frac{1}{z_1z_2}\otimes e_L.$$
    Thus the coefficient row of $\widetilde{\theta}$ is $(az_2,bz_1)$, and its kernel is generated by
    $$-z_1b\frac{\partial}{\partial z_1}+az_2\frac{\partial}{\partial z_2}.$$
    
    Define $\mathscr{K}:=\ker\widetilde{\theta}$ and $\mathscr{I}:=\operatorname{Im}\widetilde{\theta}$.
    Then $\mathcal{I}_Z:=(\mathcal{O}_X(D)\otimes L)^{-1}\otimes\mathscr{I}$ is a coherent ideal of $\mathcal{O}_X$.
    By the assumption, there is no common divisorial factor in the two coefficients of $\widetilde{\theta}$.
    It follows that $\supp(\mathcal{O}_X/\mathcal{I}_Z)$ has codimension at least two. 
    Thus $Z$ is zero-dimensional.
    
    Locally, $\widetilde{\theta}$ can be regarded as $(f,g):\mathcal{O}^{\oplus 2}\rightarrow\mathcal{O}$ with $f,g$ having no common divisorial factor.
    If they vanish at a common point, they form a regular sequence in the two-dimensional regular local ring.
    The Koszul complex gives
    $$0\longrightarrow\mathcal{O}\xrightarrow{(-g,f)}\mathcal{O}^{\oplus 2} \xrightarrow{(f,g)}(f,g) \longrightarrow 0.$$
    Consequently, $\mathscr{K}$ is locally free of rank one and the image is $ \mathcal{O}_X(D)\otimes L\otimes\mathcal{I}_Z$. 
    We have 
    $$0\longrightarrow\mathscr{K}\rightarrow TX \xrightarrow{\widetilde{\theta}}\mathcal{O}_X(D)\otimes L\otimes\mathcal{I}_Z\longrightarrow 0.$$
    
    On $X\setminus Z$, $\mathscr{K}$ is a holomorphic line subbundle of $TX$.
    If $\widetilde{\theta}=(f,g)$, its fibers are precisely the tangent directions generated by 
    $$-g\frac{\partial}{\partial z_1}+f\frac{\partial}{\partial z_2}.$$
    It therefore defines a one-dimensional singular holomorphic foliation $\mathcal{F}$.
    Note that $\mathscr{K}$ is saturated. 
    Then we have 
    $$T\mathcal{F}=\mathscr{K}=\ker\widetilde{\theta}.$$
    Define the normal line bundle by
    $$N\mathcal{F}:=(TX/T\mathcal{F})^{**}.$$
    Since $Z$ is zero-dimensional, $\mathcal{I}_Z^{**}\simeq\mathcal{O}_X$.
    Therefore, $N\mathcal{F}\simeq\mathcal{O}_X(D)\otimes L$.
    
    This completes the proof of Lemma~\ref{Lemma:The endpoint foliations}.
\end{proof}

\subsection{Mean curvature positivity}

We finally consider surfaces with no holomorphic foliation.
\begin{theorem}\label{Main Thm 3(3)}
    Let $X$ be a surface of class $\VII_0^+$ with no holomorphic foliation.
    Then $TX$ admits a Hermitian metric with positive mean curvature.
    In particular, $TX$ is RC-positive.
\end{theorem}
\begin{proof}
    By the argument at the beginning of Section~\ref{Sec:Classification 2}, $c_1^{\BC}(K_X)$ is not pseudo-effective, and 
    there exists a Gauduchon metric $\omega$ such that 
    $$\deg_\omega(-K_X)>0.$$
    Consider the Harder--Narasimhan filtration of $TX$ with respect to $\omega$
    $$0=\mathcal{E}_0\subset \mathcal{E}_1\subset \cdots \subset \mathcal{E}_l=TX,$$
    with torsion-free quotients $\mathcal{Q}_j:=\mathcal{E}_j/\mathcal{E}_{j-1}$.
    If the filtration were nontrivial, its first  term $\mathcal{E}_1$ would be a saturated rank-one subsheaf of $TX$.
    Then $\mathcal{E}_1$ is a reflexive sheaf and
    hence a line bundle.
    This subsheaf $\mathcal{E}_1\subset TX$ defines a holomorphic foliation.
    This contradicts the hypothesis.
    Hence the Harder--Narasimhan filtration is trivial and $TX$ is $\omega$-semistable.
    Li--Zhang--Zhang identify $\mu_L(E,\omega)$ with the slope of the final Harder--Narasimhan quotient (\cite[formula~(1.11)]{LZZ25}).
    This shows that 
    $$\mu_L(TX,\omega)=\mu_\omega(TX)=\frac{1}{2}\deg_\omega(-K_X)>0.$$
    Hence by Theorem~\ref{Theorem:Mean curvature positivity and HN-positivity}, $TX$ is mean curvature positive. 
    In particular, $TX$ is RC-positive.
    
\end{proof}

In summary, we establish the following result.
\begin{theorem}[=Theorem~\ref{Main Thm 3}]\label{Main Thm 3'}
    Let $X$ be a surface of class $\VII_0^+$. Then
    \begin{enumerate}
        \item $-K_X$ is (uniformly) RC-positive.
        \item If $X$ is a Kato surface, then $TX$ is
        uniformly RC-positive.
        \item If $X$ carries no holomorphic foliation, then $TX$ is mean curvature positive.
    \end{enumerate}
\end{theorem}

In particular, we have:
\begin{corollary}
    Let $X$ be a surface of class $\VII_0^+$.
    \begin{enumerate}
        \item If $b_2(X)=1$, then $TX$ is uniformly RC-positive.
        \item If $b_2(X)=2$, then $TX$ is mean curvature positive.
        If moreover $X$ admits a holomorphic foliation, then $TX$ is uniformly RC-positive.
    \end{enumerate}
\end{corollary}
\begin{proof}
    (1) Teleman proved the GSS conjecture in \cite{Tel05} for every surface of class $\VII_0^+$ with $b_2(X)=1$. 
    Thus $X$ is a Kato surface. 
    Hence $TX$ is uniformly RC-positive.
    
    (2) For $b_2(X)=2$, Brunella (\cite[Theorem~0.1]{Bru11} and \cite{Bru13}) proved that a surface carrying a foliation is a Kato surface. 
    Hence $TX$ is uniformly RC-positive, and therefore mean curvature positive.
    If $X$ carries no holomorphic foliation, then by Theorem~\ref{Main Thm 3'}, $TX$ is mean curvature positive.
\end{proof}


\begin{thebibliography}{99}
    
    \bibitem[AD08]{AD08} Vestislav Apostolov and Georges Dloussky, \textit{Bihermitian metrics on Hopf surfaces}, Math. Res. Lett. {\bf 15} (2008), no.~5, 827--839.
    
    \bibitem[AD23]{AD23} Vestislav Apostolov and Georges Dloussky, \textit{Twisted differentials and Lee classes of locally conformally symplectic complex surfaces}, Math. Z. {\bf 303} (2023), no.~3, Paper No. 76, 33 pp.
    
    \bibitem[AK03]{AK03} Carolina Araujo and J\'anos Koll\'ar, \textit{Rational curves on varieties}, in Higher dimensional varieties and rational points (Budapest, 2001), 13--68, Bolyai Soc. Math. Stud., 12, Springer, Berlin, 2003.
    
    \bibitem[BPV84]{BPV84} W.~P. Barth, C.~A.~M. Peters and A.~J.~H.~M. Van~de~Ven, \textit{Compact complex surfaces}, Ergebnisse der Mathematik und ihrer Grenzgebiete (3), 4, Springer, Berlin, 1984.
    
    \bibitem[Bog76]{Bog76} Fedor Alekseevich Bogomolov, \textit{Classification of surfaces of class VII0 with $b\sb{2}=0$}, Izv. Akad. Nauk SSSR Ser. Mat. {\bf 40} (1976), no.~2, 273--288, 469.
    
    \bibitem[Bog82]{Bog82} Fedor Alekseevich Bogomolov, \textit{Surfaces of class ${\rm VII}\sb{0}$\ and affine geometry}, Izv. Akad. Nauk SSSR Ser. Mat. {\bf 46} (1982), no.~4, 710--761, 896.
    
    \bibitem[Bro24]{Bro24} Garrett M. Brown, \textit{The sign of scalar curvature on K\"ahler blowups}, Preprint, arXiv:2405.12189 (2024).
    
    \bibitem[Bru11]{Bru11} Marco Brunella, \textit{On a class of foliated non-K\"ahlerian compact complex surfaces}, Tohoku Math. J. (2) {\bf 63} (2011), no.~3, 441--460.
    
    \bibitem[Bru13]{Bru13} Marco Brunella, \textit{Corrigendum to ``On a class of foliated non-K\"ahlerian compact complex surfaces'' [MR2851106]}, Tohoku Math. J. (2) {\bf 65} (2013), no.~4, 607--608.
    \bibitem[CDP15]{CDP15}F.~Campana, J.-P.~Demailly, and Th.~Peternell, \textit{Rationally connected manifolds and semipositivity of the Ricci curvature}, in \emph{Recent advances in algebraic geometry}, London Math. Soc. Lecture Note Ser., vol. 417, Cambridge Univ. Press, Cambridge, 2015, pp. 71--91. MR3380444
    \bibitem[CTW19]{CTW19} Jianchun Chu, Valentino Tosatti, and Ben Weinkove, \textit{The Monge-Amp\`ere equation for non-integrable almost complex structures}, J. Eur. Math. Soc. (JEMS) {\bf 21} (2019), no.~7, 1949--1984.

    \bibitem[Dem]{Dem} Jean-Pierre Demailly, \textit{Complex analytic and differential geometry}, 2012, available at \url{https://www-fourier.univ-grenoble-alpes.fr/~demailly/manuscripts/agbook.pdf}.
    
    \bibitem[Dlo88-a]{Dlo88-a} Georges Dloussky, \textit{Sur la classification des germes d'applications holomorphes contractantes}, Math. Ann. {\bf 280} (1988), no.~4, 649--661.
    
    \bibitem[Dlo88-b]{Dlo88-b} Georges Dloussky, \textit{Une construction \'el\'ementaire des surfaces d'Inoue-Hirzebruch}, Math. Ann. {\bf 280} (1988), no.~4, 663--682.
    
    \bibitem[DK98]{DK98} Georges Dloussky and Franz Kohler, \textit{Classification of singular germs of mappings and deformations of compact surfaces of class VII$_0$}, Ann. Polon. Math. {\bf 70} (1998), 49--83.
    
    \bibitem[DO99]{DO99} Georges Dloussky and Karl Oeljeklaus, \textit{Vector fields and foliations on compact surfaces of class $\rm VII_0$}, Ann. Inst. Fourier (Grenoble) {\bf 49} (1999), no.~5, 1503--1545.
    
    \bibitem[Eno81]{Eno81} Ichiro Enoki, \textit{Surfaces of class ${\rm VII}\sb{0}$\ with curves}, Tohoku Math. J. (2) {\bf 33} (1981), no.~4, 453--492.
    
    \bibitem[GZ11]{GZ11} Ning Gan and Xiangyu Zhou, \textit{Structure of vector bundles over general Hopf manifolds}, Chinese Ann. Math. Ser. A {\bf 32} (2011), no.~1, 115--120.
    
    \bibitem[GZ20]{GZ20} Ning Gan and Xiangyu Zhou, \textit{The structure of vector bundles on non-primary Hopf manifolds}, Chinese Ann. Math. Ser. B {\bf 41} (2020), no.~6, 929--938.
    
    \bibitem[Gau77]{Gau77} Paul Gauduchon, \textit{Le th\'eor\`eme de l'excentricit\'e{} nulle}, C. R. Acad. Sci. Paris S\'er. A-B {\bf 285} (1977), no.~5, {\rm A}387--{\rm A}390.
        
    \bibitem[GO98]{GO98} Paul Gauduchon and Liviu Ornea, \textit{Locally conformally K\"ahler metrics on Hopf surfaces}, Ann. Inst. Fourier (Grenoble) {\bf 48} (1998), no.~4, 1107--1127.    
    
    \bibitem[Has93]{Has93} Keizo Hasegawa, \textit{Deformations and diffeomorphism types of Hopf manifolds}, Illinois J. Math. {\bf 37} (1993), no.~4, 643--651.
    
    \bibitem[Has23]{Has23} Keizo Hasegawa, \textit{A note on locally conformally K\"ahler structures and small deformations of Hopf manifolds}, Preprint, arXiv:2305.10768v2 (2023).
    
    \bibitem[Hit75]{Hit75} Nigel Hitchin, \textit{On the curvature of rational surfaces}, in Differential geometry (Proc. Sympos. Pure Math., Vol. XXVII, Part 2, Stanford Univ., Stanford, Calif., 1973), pp. 65--80, Proc. Sympos. Pure Math., Vol. XXVII, Part 2, Amer. Math. Soc., Providence, RI, 1975.
    
    \bibitem[Ino77]{Ino77} Masahisa Inoue, \textit{New surfaces with no meromorphic functions. II}, in {\it Complex analysis and algebraic geometry}, pp. 91--106, Iwanami Shoten Publishers, Tokyo, 1977.
    
    \bibitem[Kat77]{Kat77} Masahide Kato, \textit{Compact complex manifolds containing ``global''\ spherical shells. I}, in Proceedings of the International Symposium on Algebraic Geometry (Kyoto Univ., Kyoto, 1977), pp. 45--84.
    
    \bibitem[Kol96]{Kol96} J\'anos Koll\'ar, \textit{Rational curves on algebraic varieties}, Ergebnisse der Mathematik und ihrer Grenzgebiete. 3. Folge. A Series of Modern Surveys in Mathematics, 32, Springer, Berlin, 1996.
    
    \bibitem[KS26]{KS26} Nikon Kurnosov and Calum Spicer, \textit{Class VII surfaces with $b_2=3$ and two foliations are Kato}, Preprint, arXiv:2608.29047 (2026).
    
    \bibitem[Lee13]{Lee13} John M. Lee, \textit{Introduction to smooth manifolds}, second edition, Graduate Texts in Mathematics, 218, Springer, New York, 2013.

    \bibitem[LZZ25]{LZZ25} Chao Li, Chuanjing Zhang, and Xi Zhang, \textit{Mean curvature positivity and rational connectedness}, Adv. Math. {\bf 483} (2025), Paper No. 110673, 43 pp.
    
    \bibitem[LY87]{LY87} Jun Li and Shing-Tung Yau, \textit{Hermitian-Yang-Mills connection on non-K\"ahler manifolds}, in  Mathematical aspects of string theory (San Diego, Calif., 1986), 560--573, Adv. Ser. Math. Phys., 1, World Sci. Publishing, Singapore, 1987.
    
    \bibitem[LYZ90]{LYZ90}Jun Li, Shing-Tung Yau, and Fangyang Zheng, \textit{A simple proof of Bogomolov's theorem on class ${\rm VII}_0$ surfaces with $b_2=0$}, Illinois J. Math. {\bf 34} (1990), no.~2, 217--220.

    \bibitem[LY17]{LY17} Kefeng Liu and Xiaokui Yang, \textit{Ricci curvatures on Hermitian manifolds}, Trans. Amer. Math. Soc. {\bf 369} (2017), no.~7, 5157--5196.
    
    \bibitem[Mic82]{Mic82} Marie-Louise Michelsohn, \textit{On the existence of special metrics in complex geometry}, Acta Math. {\bf 149} (1982), no.~3-4, 261--295.
    
    \bibitem[Nak84]{Nak84} Iku Nakamura, \textit{On surfaces of class ${\rm VII}_0$ with curves}, Invent. Math. {\bf 78} (1984), no.~3, 393--443.
    
    \bibitem[Tel94]{Tel94} Andrei Teleman, \textit{Projectively flat surfaces and Bogomolov's theorem on class ${\rm VII}_0$ surfaces}, Internat. J. Math. {\bf 5} (1994), no.~2, 253--264.
    
    \bibitem[Tel05]{Tel05} Andrei Teleman, \textit{Donaldson theory on non-K\"ahlerian surfaces and class VII surfaces with $b_2=1$}, Invent. Math. {\bf 162} (2005), no.~3, 493--521.
    
    \bibitem[Tel06]{Tel06} Andrei Teleman, \textit{The pseudo-effective cone of a non-K\"ahlerian surface and applications}, Math. Ann. {\bf 335} (2006), no.~4, 965--989.
    
    \bibitem[Y1]{Yan17} Xiaokui Yang, \textit{Big vector bundles and complex manifolds with semi-positive tangent bundles}, Math. Ann. {\bf 367} (2017), no.~1-2, 251--282.
    
    \bibitem[Y2]{Yan18} Xiaokui Yang, \textit{RC-positivity, rational connectedness and Yau's conjecture}, Camb. J. Math. {\bf 6} (2018), no.~2, 183--212.
    
    \bibitem[Y3]{Yan19} Xiaokui Yang, \textit{A partial converse to the Andreotti-Grauert theorem}, Compos. Math. {\bf 155} (2019), no.~1, 89--99.
    
    \bibitem[Y4]{Yan20} Xiaokui Yang, \textit{RC-positive metrics on rationally connected manifolds}, Forum Math. Sigma {\bf 8} (2020), Paper No. e53, 19 pp.
    
    \bibitem[Yau82]{Yau82} Shing-Tung Yau, \textit{Problem section}, in Seminar on Differential Geometry, pp. 669--706, Ann. of Math. Stud., No. 102, Princeton Univ. Press, Princeton, NJ, 1982.
    
    \bibitem[Z]{Zha26} Shiyu Zhang, \textit{Positive holomorphic sectional curvature on rational surfaces}, Preprint, arXiv:2606.23333 (2026).
    
    \bibitem[Zho09]{Zho09} Xiangyu Zhou, \textit{Some results about holomorphic vector bundles over general Hopf manifolds}, Sci. China Ser. A {\bf 52} (2009), no.~12, 2863--2866.


\end{thebibliography}
\end{document}